\documentclass[twoside,british,12pt]{article}

\usepackage{babel}

\usepackage{amsmath,amsfonts,amssymb,amscd,amsthm,amsbsy}
\usepackage{fancyhdr,latexsym,epsfig}

\usepackage[utf8]{inputenc}
\usepackage[T1]{fontenc}
\usepackage{url}

\newtheorem{thm}{Theorem}[section]
\newtheorem{lem}[thm]{Lemma}
\newtheorem{prop}[thm]{Proposition}
\newtheorem{cor}[thm]{Corollary}

\theoremstyle{definition}
\newtheorem{rem}[thm]{Remark}
\newtheorem{rems}[thm]{Remarks}
\newtheorem{exo}{Example}
 
\newcommand{\reth}[1]{Theorem~\protect\ref{th:#1}}
\newcommand{\relem}[1]{Lemma~\protect\ref{lem:#1}}
\newcommand{\repr}[1]{Proposition~\protect\ref{prop:#1}}
\newcommand{\reco}[1]{Corollary~\protect\ref{cor:#1}}
\newcommand{\resec}[1]{Section~\protect\ref{sec:#1}}
\newcommand{\rerem}[1]{Remark~\protect\ref{rem:#1}}
\newcommand{\rerems}[1]{Remarks~\protect\ref{rems:#1}}

\newcommand{\reqref}[1]{(\protect\ref{eq:#1})}

\newcommand{\rex}[1]{Example~\protect\ref{ex:#1}}

\newcommand{\brak}[1]{\ensuremath{\left\{ #1 \right\}}}
\newcommand{\ang}[1]{\ensuremath{\langle #1\rangle}}

\newcommand{\setangl}[2]{\ensuremath{\ang{\left. #1 \,\right\rvert \, #2}}}
\newcommand{\setl}[2]{\ensuremath{\brak{\left. #1 \,\right\rvert \, #2}}}

\newcommand{\N}{\ensuremath{\mathbb N}}
\newcommand{\Z}{\ensuremath{\mathbb Z}}
\newcommand{\D}{\ensuremath{\mathbb D}}
\newcommand{\T}{\ensuremath{\mathbb{T}^2}}
\newcommand{\K}{\ensuremath{\mathbb{K}^2}}
\newcommand{\SP}{\ensuremath{\mathbb{S}^{2}}}
\newcommand{\FF}{\ensuremath{\mathbb F}}
\newcommand{\RP}{\ensuremath{\mathbb{R}P^{2}}}
\renewcommand{\epsilon}{\ensuremath{\varepsilon}}
\renewcommand{\phi}{\ensuremath{\varphi}}
\renewcommand{\to}{\ensuremath{\longrightarrow}}

\renewcommand{\ker}[1]{\ensuremath{\operatorname{\text{Ker}}\left({#1}\right)}}
\newcommand{\im}[1]{\ensuremath{\operatorname{\text{Im}}\left({#1}\right)}}
\newcommand{\aut}[1]{\ensuremath{\operatorname{\text{Aut}}\left({#1}\right)}}

\usepackage{enumitem}
\setlist[enumerate,1]{leftmargin=0em,labelwidth=1em,wide,itemsep=-3pt,labelsep=0.5em,topsep=0.3em,align=left,labelindent*=0em,label=\textit{(\alph*)}}
\setlist[enumerate,2]{leftmargin=5em,labelwidth=1em,wide,itemsep=0em,labelsep=0.5em,topsep=0.3em,partopsep=0.3em,align=left,label=\textit{(\roman*)}}

\usepackage[
bookmarks=true,
breaklinks=true,
bookmarksnumbered = true,
colorlinks= true,
urlcolor= green,
anchorcolor = yellow,
citecolor=blue,
]{hyperref}                   

\usepackage[x11names]{xcolor}

\begin{document}

\pagestyle{myheadings}

\markboth{P.~Bellingeri, D.~Gon\c{c}alves and J.~Guaschi}{Lower central series, surface braid groups,
surjections and permutations}

\title{Lower central series, surface braid groups, surjections and permutations}

\author{Paolo Bellingeri\thanks{Normandie Univ., UNICAEN, CNRS, Laboratoire de Math\'ematiques Nicolas Oresme UMR CNRS~\textup{6139}, CS 14032, 14032 Cedex Cedex 5, France, \url{paolo.bellingeri@unicaen.fr}, \url{john.guaschi@unicaen.fr}}, Daciberg Lima Gon\c{c}alves\thanks{Departamento de Matem\'atica - IME-USP, Rua~do~Mat\~ao~1010~CEP:~05508-090 - S\~ao Paulo - SP - Brazil, \url{dlgoncal@ime.usp.br}}\, and John Guaschi\footnotemark[1]}

\date{\today}

\maketitle

\begin{abstract} 
Generalising previous results on classical braid groups  by Artin and Lin, we determine the values of $m,n\in \N$ for which there exists a surjection between the $n$- and $m$-string braid groups of an orientable surface without boundary. This result is essentially based on specific properties of their lower central series, and the proof is completely combinatorial. We provide similar but partial results in the case of orientable surfaces with boundary components and of non-orientable surfaces without boundary. We give also several results about the classification of different representations of surface braid groups in symmetric groups.
\end{abstract}

\begingroup
\renewcommand{\thefootnote}{}
\footnotetext{2010 AMS Mathematics Subject Classification: 20F14 20F14 Derived series, central series, and generalizations;  20F36 20F14 Derived series, central series, and generalizations; 20B15 Primitive groups. Keywords: surface braid groups,  lower central series, symmetric groups, representations.}
\endgroup 

\maketitle


\section{Introduction}\label{sec:intro}

In 1947, E.~Artin published two seminal papers in the Annals of Mathematics, sometimes considered as the foundation of the theory of braid groups. The paper~\cite{Art1} is devoted to a determining a presentation for the braid group $B_n$ on $n$ strings, and its interpretation in terms of automorphisms of the free group of rank $n$, while the subject of~\cite{Art2} is the study of possible homomorphisms from $B_n$ to the symmetric group $S_n$ on $n$ letters. The main result of~\cite{Art2} is the description of all \emph{transitive} homomorphisms (see \resec{other} for the definition) between $B_n$ and $S_n$. Artin considered this characterisation to be the first step in determining the group of automorphisms $\aut{B_{n}}$ of $B_{n}$, a solution of which was given in~\cite{DG}. In~\cite{L0}, Lin generalised Artin's results by characterising the homomorphisms between $B_n$ and $S_m$ and between $B_n$ and $B_m$, for all $n>m$ (see~\cite{L} for a proof of these results and a survey of this topic). Other (partial) results for homomorphisms between 
$B_n$ and $B_m$ with $n < m$ were recently obtained in~\cite{BM} and in~\cite{Ca}.

The main subject of this paper is surface braid groups and the existence of surjective homomorphisms between them. These groups generalise both Artin's braid groups and fundamental groups of surfaces. As well as their geometric interpretation, they may be defined in terms of fundamental groups of configuration spaces as follows~\cite{FoN}. Let $\Sigma$ be a compact, connected surface, with or without boundary, orientable or non orientable, and let $\FF_n(\Sigma)=\Sigma^n \setminus \Delta$, where $\Delta$ is the set of $n$-tuples $(x_1, \ldots, x_n)$ of elements of $\Sigma$ for which $x_i=x_j$ for some $1\leq i,j\leq n$, where $i \not= j$. The
fundamental group $\pi_1(\FF_n(\Sigma))$ is called the \emph{pure braid group} on $n$ strings of $\Sigma$ and shall be denoted by $P_n(\Sigma)$. The symmetric group $S_n$ acts freely on $\FF_n(\Sigma)$ by permutation of coordinates, and the fundamental group $\pi_1(\FF_n(\Sigma)/S_n)$ of the resulting quotient space, denoted by $B_n(\Sigma)$, is the \emph{braid group} on $n$ strings of $\Sigma$. Further, $\FF_n(\Sigma)$ is a regular $n!$-fold covering of $\FF_n(\Sigma)/S_n$, from which we obtain the following short exact sequence:
\begin{equation}\label{eq:permutation}
1\to P_n(\Sigma) \to B_n(\Sigma) \to S_{n}\to 1.
\end{equation}
If $\Sigma$ is the $2$-disc $\D^{2}$, it is well known that $B_{n}(\D^2)\cong B_{n}$ and that $P_{n}(\D^2)\cong P_{n}$.

The fibration of configuration spaces of a surface $\Sigma$ without boundary defined by Fadell and Neuwirth~\cite{FN} gives rise to an exact sequence involving the pure braid groups of $\Sigma$, from which one may see that the forgetful homomorphism from $P_{n+m}(\Sigma)$ to $P_n(\Sigma)$ given geometrically by forgetting $m$ strings is well defined and is a surjection. Recently, it was stated in~\cite[Theorem~1.1]{Ch} that forgetful homomorphisms are `essentially' the only possible surjections between pure braid groups of orientable surfaces, and it was conjectured that if $\Sigma$ is a compact, orientable surface of genus $g>1$, with or without boundary, and if $m,n\in \N$, where $m\neq n$, then there is no surjective homomorphism from $B_{n}(\Sigma)$ to $B_{m}(\Sigma)$~\cite[Conjecture~1.3]{Ch}. One of the aims of this paper is to study this problem for compact, connected surfaces, with or without boundary, orientable or non orientable. We summarise our main results in this direction as follows. In what follows, $\SP$ (resp.\ $\RP$) will denote the $2$-sphere (resp.\ the real projective plane), $\T$ (resp.\ $\K$) will denote the $2$-torus (resp.\ the Klein bottle), $\Sigma_{g}$ (resp.\ $\Sigma_{g,b}$) will be a compact, connected, orientable surface of genus $g\geq 0$ without boundary (resp.\ with $b\geq 1$ boundary components), and $U_{g}$ will be a compact, connected, non-orientable surface of genus $g\geq 1$ without boundary (in other words, $U_{g}$ is the connected sum of $g$ projective planes).


\begin{thm}\label{th:gensurj}
Let $m,n\in \N$ be such that $m\neq n$. 
\begin{enumerate}[wide=0em,labelsep=0.2em]
\item\label{it:main1a}
\begin{enumerate}[wide=1em]
\item\label{it:main1ai} There is a surjective homomorphism from $B_{n}(\SP)$ to $B_{m}(\SP)$ if and only if $m\in \brak{1,2}$ and $n>m$.

\item\label{it:main1aii} If $g\geq 1$, there is a surjective homomorphism from $B_{n}(\Sigma_{g})$ to $B_{m}(\Sigma_{g})$ if and only if $m=g=1$. 
\end{enumerate}


\item\label{it:main1b} Let $g\geq 1$, and let $\Sigma$ be either $\Sigma_{g,b}$, where $b\geq 1$, or $U_{g+1}$. 
Suppose that one of the following conditions holds:
\begin{enumerate}
\item\label{it:main1bi} $n<m$ and $n\in \brak{1,2}$.
\item\label{it:main1bii} $n>m$ and $m\in \brak{1,2}$.
\item\label{it:main1biii} $n>m\geq 3$ and $n\neq 4$.
\end{enumerate}
Then there is no surjective homomorphism from $B_{n}(\Sigma)$ to $B_{m}(\Sigma)$.
\end{enumerate}
\end{thm}

Parts~\ref{it:main1a}\ref{it:main1ai} and~\ref{it:main1a}\ref{it:main1aii} of \reth{gensurj} will be proved in \resec{sphere} and \resec{suror} respectively, and 
part~\ref{it:main1b} will be proved in \resec{orboundary} in the orientable case, and in \resec{nonor} in the non-orientable case. For the case of the projective plane, in \reth{rp}, using the knowledge of the torsion of its braid groups, we will obtain results that are slightly stronger than those of \reth{gensurj}\ref{it:main1b}, notably with respect to the case where $n<m$. 
\reth{gensurj} proves~\cite[Conjecture~1.3]{Ch} completely in the case where the surface is orientable and without boundary, and partially in the case where the surface orientable with boundary, or non-orientable and without boundary. The cases not covered by the conditions~\ref{it:main1bi}--\ref{it:main1biii} of part~\ref{it:main1b} are likely to be difficult. Note that even in the case of the Artin braid groups, the question of whether there exists a surjective homomorphism from $B_{n}$ onto $B_{m}$ remains open in many of these cases.
As a consequence of \reth{gensurj} and some basic facts about the lower central series of surface braid groups, we give an elementary proof of~\cite[Theorem~1.2]{Ch} in \reco{pure}, and we generalise the result of this corollary to the case of orientable surfaces with boundary (\reco{pureboundary}), and to the non-orientable case (\reco{purenor}). In the cases of the sphere and real projective plane, the techniques are somewhat different to those used for other surfaces, since their braid groups have torsion~\cite{FV,VB}.

Another interesting and open problem is the study of possible surjective homomorphisms between braid groups of different surfaces. One important case occurs when the domain is a braid group of a non-orientable surface $U_{g}$, and the target is a braid group of the orientable double covering $\Sigma_{g-1}$. It is known that there exists a natural injection on the level of configuration spaces that induces an injective homomorphism between $B_{n}(U_{g})$ and $B_{2n}(\Sigma_{g-1})$~\cite{GGjlms}. In \resec{nonor}, we prove the following result concerning surjections when the number of strings is the same.

\begin{prop}\label{prop:ornor}
Let $n,g\ge1$. Then there exists a surjective homomorphism of $B_n(U_g)$ onto $B_n(\Sigma_{g-1})$ if and only if $g=1$ and $n\in \brak{1,2}$.
\end{prop}

Together with surjections between surface braid groups, another of our aims is to characterise homomorphisms between surface braid groups and symmetric groups following the approaches of~\cite{Iv,L0,L}. One of the main results of~\cite{L} is the following.
 
\begin{thm}[{\cite[Theorem~A]{L}}]\label{th:Lin}
Let $n>m\ge 3$ and $n\not=4$. Any homomorphism $\phi\colon\thinspace B_n \to S_m$ is cyclic, i.e.\ $\phi(B_n)$ is a cyclic group.
\end{thm}

This implies that if $n>m\ge 3$ and $n\not=4$, there is no surjective homomorphism from $B_n$ onto $S_m$. We shall show that \reth{Lin} also holds for braid groups of compact surfaces without boundary.

\begin{thm}\label{th:LinSurfcomb}
Let $n>m\ge 2$, let $g \ge 0$, and let $\Sigma$ be either $\Sigma_{g}$ or $U_{g+1}$. Then there is a surjective homomorphism from $B_n(\Sigma)$ onto $S_m$ if and only if either $m=2$, or $(n,m)=(4,3)$.  
\end{thm}

If $g\geq 1$ (resp.\ $g=0$) and $\Sigma=\Sigma_{g}$, the statement of \reth{LinSurfcomb} will be proved in \resec{suror} (resp.\ in \resec{sphere}), while in the case $g\geq 0$ and  $\Sigma=U_{g+1}$, the result will be proved in \resec{nonor}. 
\medskip

Let $g \geq 0$,
and let $n>m \geq 1$. We recall that a representation $\rho_{n,m}\colon\thinspace  B_n(\Sigma_g) \to S_m$ is said to be \emph{transitive} if the action of the image $\im{\rho_{n,m}}$ of $\rho_{n,m}$ on the set $\brak{1,\ldots, m}$ is transitive and is \emph{primitive} if the only partitions of this set that are left invariant by the action of  $\im{\rho_{n,m}}$ are the set itself, or the partition consisting of singletons.  By abuse of notation, we say that a subgroup  of $S_m$ is primitive if its action on the set $\brak{1,\ldots, m}$ is primitive.
Notice that, if $m>2$, a primitive representation is clearly transitive. 
 \medskip

 Inspired by Artin's characterisation of (transitive) homomorphisms between $B_n$ and $S_{n}$, Ivanov determined all of the homomorphisms between $B_n(\Sigma_{g,b})$ and $S_{n}$, but under the stronger assumption that the homomorphisms are primitive~\cite[Theorem~1]{Iv}. We prove the following theorem for homomorphisms between $B_n(\Sigma_{g}) $ and $ S_m$ when $n>m$. This result may also be compared with the classification of the homomorphisms between $B_n$ and $S_{m}$, where $n>m$, given in \reth{LinTrans}.


\begin{thm}\label{th:LinSurfaceTrans}
Let $n>m\geq 2$, and let $g\geq 1$. There exists a primitive representation
$\rho_{n,m}\colon\thinspace $ $B_n(\Sigma_{g}) \to S_m$ if and only if $m$ is prime. This being the case, one of the following statements holds:
\begin{enumerate}
\item\label{it:LinSurfaceTransa} the image $\im{\rho_{n,m}}$ of $\rho_{n,m}$ is generated by an $m$-cycle, unless $m=2$, in which case $\im{\rho_{n,2}}$ can also be equal to $\brak{\operatorname{Id}}$.


\item\label{it:LinSurfaceTransb} $n=4$ and $m=3$, and up to a suitable renumbering of the elements of the set $\brak{1, 2, 3}$,
$\rho_{4,3}(\sigma_1)=\rho_{4,3}(\sigma_3)=(1,2)$, $\rho_{4,3}(\sigma_2)=(2,3)$, and for all $1\leq i\leq g$, the permutations $\rho_{4,3}(a_i)$ and $\rho_{4,3}(b_{i})$ are trivial, where $\brak{\sigma_1, \sigma_2,\sigma_3, a_1,b_1, \ldots, a_g,b_g}$ is the generating set of $B_4(\Sigma_{g})$ given in the statement of \reth{presbng}.
\end{enumerate}
\end{thm}

More information about arbitrary (not necessarily primitive) representations of $B_n(\Sigma_{g})$ in $S_m$ is given in \repr{contraints} and the examples that follow it.

The rest of this paper is divided into four sections. Sections~\ref{sec:closed} and~\ref{sec:orboundary} deal with the braid groups of compact, orientable surfaces without boundary and with boundary respectively, \resec{nonor} is devoted to the braid groups of compact, non-orientable surfaces without boundary. In each of these sections, we give a presentation of the braid groups in question, we recall some known results about their lower central series and whether they are residually nilpotent or not, and we prove the relevant parts of Theorems~\ref{th:gensurj} and~\ref{th:LinSurfcomb}. In \resec{other} we explore representations of surface braid groups in symmetric groups. In particular, in \reth{LinSurfaceTrans}, we classify primitive representations, a result that was more or less implicitly expected in~\cite{Iv}, and in \repr{contraints} we give some constraints on general (non-primitive) homomorphisms, and we answer a question of~\cite{Iv} by providing some examples of transitive, non-primitive, non-Abelian representations.

In this paper, we do not discuss the braid groups of non-orientable surface with boundary components. This choice is motivated by two different considerations, first that these groups have rarely been studied in the literature, and secondly, that the techniques used in the case of non-orientable surfaces without boundary apply almost verbatim to the case with boundary. This is in contrast with the orientable case, where the lower central series is a stronger tool in the case without boundary than in the case with boundary.

\subsubsection*{Acknowledgements}

The second and third authors of this paper were partitially supported by the CNRS/FAPESP PRC project n\textsuperscript{o}~275209 (France) and n\textsuperscript{o}~2016/50354-1 (Brazil). The second  author is also partially supported by the FAPESP Projeto Tem\'atico `Topologia Alg\'ebrica, Geom\'etrica e Diferencial' n\textsuperscript{o}~2016/24707-4 (Brazil).

\section{Orientable surfaces without boundary}\label{sec:closed}

In \resec{orlcs}, we start by recalling a presentation of the braid groups of compact, orientable surfaces without boundary, as well as some facts about their lower central series. In \resec{suror}, we generalise certain results of~\cite{GGzeit} about the minimal number of generators of these groups, and we prove Theorems~\ref{th:gensurj}\ref{it:main1a}\ref{it:main1aii} and \ref{th:LinSurfcomb} in the case where $\Sigma=\Sigma_{g}$, with $g\geq 1$. In \resec{sphere}, we prove \reth{gensurj}\ref{it:main1a}\ref{it:main1ai} and \reth{LinSurfcomb} in the case $g=0$, which is that of the sphere.

\subsection{Presentations and the lower central series of surface braid groups}\label{sec:orlcs}
 
In this paper, many of our techniques will be combinatorial and will make use of the lower central series of surface (pure) braid groups. Given a group $G$, recall that the \emph{lower central series} of $G$ is given by $\brak{\Gamma_i(G)}_{i\in \N}$, where $G=\Gamma_1(G)$, and $\Gamma_i(G)=[G,\Gamma_{i-1}(G)]$ for all $i\geq 2$. We thus have a filtration $\Gamma_1(G) \supseteq \Gamma_2(G) \supseteq \cdots$.
The group $G$ is said to be \emph{perfect} if $G=\Gamma_2(G)$. We shall denote the \emph{Abelianisation} $\Gamma_1(G)/\Gamma_2(G)$ of $G$ by $G_{\text{Ab}}$. If $\mathcal{P}$ is a group-theoretic property, let $\mathcal{FP}$ denote the class of groups that possess property $\mathcal{P}$. Following P.~Hall, $G$ is said to be \emph{residually $\mathcal{P}$} if for any (non-trivial) element $x\in G$, there exists a group $H$ possessing property $\mathcal{P}$ and a surjective homomorphism $\phi\colon\thinspace G \to H$ such that  $\phi(x) \not=1$. It is well known that a group $G$ is residually nilpotent if and only if $\bigcap_{i \ge 1}\Gamma_i(G)=\{ 1\}$. 
%
The lower central series of the Artin braid groups is well known.

\begin{prop}[see~\cite{BGeG,L}]\label{prop:braidlcs}
If $n\ge 3$,
$\Gamma_1(B_n)/\Gamma_2(B_n) \cong \Z$, and $\Gamma_2(B_n)=\Gamma_3(B_n)$.
\end{prop}

\repr{braidlcs} also holds trivially if $n=2$ since $B_2\cong \Z$ (and therefore $\Gamma_2(B_n)= \Gamma_3(B_n)=1$). Using~\cite{FR,Ko} and the fact that $P_2$ is isomorphic to $\Z$, we see that the group $P_n$ is residually (torsion-free) nilpotent for all $n\ge 2$.
 
We recall a presentation of $B_n(\Sigma_g)$ for $g\geq 1$. 

\begin{thm}[{\cite[Theorem~6]{BGeG}}]\label{th:presbng}
Let $g,n\in\N$. Then $B_n(\Sigma_g)$ admits the following group presentation:
\begin{itemize}
\item \textbf{generators:} $a_1, b_1, \ldots, a_g, b_g, \sigma_1, \ldots, \sigma_{n-1}$.
\item \textbf{relations:}
\begin{gather}
\text{$\sigma_i\sigma_j=\sigma_j\sigma_i$ if $\lvert i-j \rvert \geq
2$}\label{eq:artin1}\\
\text{$\sigma_i\sigma_{i+1}\sigma_i= \sigma_{i+1}\sigma_i \sigma_{i+1}$
for all $1\leq i\leq n-2$}\label{eq:artin2}\\
\text{$c_i\sigma_j= \sigma_j c_i$ for all $j\geq 2$,   $c_i=a_i$ or
  $b_i$ and $i=1, \ldots, g$}\label{eq:cs}\\
\text{$c_i \sigma_1 c_i \sigma_1= \sigma_1 c_i \sigma_1 c_i$   for
  $c_i=a_i$ or $b_i$ and $i=1, \ldots, g$}\label{eq:cici}\\
\text{$a_i \sigma_1 b_i = \sigma_1 b_i \sigma_1 a_i \sigma_1$ for  $i=1, \ldots, g$}\label{eq:ab}\\
\text{$c_i \sigma_1^{-1}  c_j \sigma_1=\sigma_1^{-1} c_j \sigma_1 c_i$
  for $c_i=a_i$ or $b_i$, $c_j=a_j$ or $b_j$ and $1\le j<i\le g$}\label{eq:cicj}\\
\text{$\prod_{i=1}^g [a_i^{-1},b_i]= \sigma_1\cdots \sigma_{n-2} \sigma_{n-1}^2 \sigma_{n-2} \cdots \sigma_1$}\label{eq:tot}.
\end{gather}
\end{itemize}
\end{thm}

Throughout this paper, relations~(\ref{eq:artin1}) and~(\ref{eq:artin2}) will be referred to as the \emph{braid} or \emph{Artin relations}. Observe that if we take $g=0$ in the presentation of \reth{presbng}, we obtain the presentation of $B_{n}(\SP)$ due to Fadell and Van Buskirk~\cite{FV}, the relations being the braid relations and the `surface relation':
\begin{equation}\label{eq:totsph}
\sigma_1\cdots \sigma_{n-2} \sigma_{n-1}^2 \sigma_{n-2} \cdots \sigma_1=1,
\end{equation}
so \reth{presbng} is also valid in this case.

If $g\geq 1$, the lower central series of the braid groups of $\Sigma_{g}$ were studied in~\cite{BGeG}. The 
statement of the following theorem contains some of the results of that paper, and provides some minor improvements, notably in the case $n=2$.

\begin{thm}\label{th:gam3closed} 
Let $g,n\geq 1$. Then:  
\begin{enumerate}
\item\label{it:gam12g} $\Gamma_1(B_n(\Sigma_g))/\Gamma_2(B_n(\Sigma_g)) \cong 
\begin{cases}
\Z^{2g} & \text{if $n=1$}\\
\Z^{2g} \oplus \Z_2 & \text{if $n\geq 2$.}
\end{cases}$

\item\label{it:gam23g}\begin{enumerate}[wide=0em,labelsep=0.2em]
\item\label{it:gam23gn} $\Gamma_2(B_n(\Sigma_g))/\Gamma_3(B_n(\Sigma_g)) \cong 
\begin{cases}
 \Z^{g(2g-1)-1} & \text{if $n=1$}\\
\Z_{n-1+g} & \text{if $n\geq 3$.}
\end{cases}$

\item\label{it:gam23gn2} If $n=2$, $\Gamma_2(B_2(\T))/\Gamma_3(B_2(\T))\cong \Z_{2}^{3}$, and if $g>1$, $\Gamma_2(B_2(\Sigma_g))/\Gamma_3(B_2(\Sigma_g))$ is a non-trivial quotient of $\Z_2^{2g} \oplus \Z_{g+1}$.
\end{enumerate}

\item\label{it:gam3g} $\Gamma_3(B_n(\Sigma_g))=\Gamma_4(B_n(\Sigma_g))$ if and only if $n\geq 3$. Moreover $\Gamma_3(B_n(\Sigma_{g}))$ is perfect if and only if $n\ge 5$.
\item\label{it:gam4g} The group $B_n(\Sigma_g)$ is residually nilpotent if and only if $n\leq 2$.
\end{enumerate}
\end{thm}


Parts~\ref{it:gam12g} and~\ref{it:gam23g} of \reth{gam3closed} imply that the braid groups of orientable surfaces without boundary may be distinguished by their lower central series (and indeed by the first two lower central series quotients). A presentation of the group $B_n(\Sigma_{g})/\Gamma_3(B_n(\Sigma_{g}))$ was given in~\cite[eq.~(10)]{BGeG} and may be found in \rex{exo1}. Many of the statements of this theorem were proved in~\cite[Theorem~1]{BGeG} in the case $n\geq 3$, and may be deduced from~\cite{LA} in the case $n=1$. More information about the lower central series quotients of $B_1(\Sigma_{g})$ may be found in~\cite{LA}. Taking into account these papers, at the end of this section, we prove \reth{gam3closed}. We first give some preliminary results and properties regarding the remaining parts of the statement, notably in the case where $n=2$. If $g=1$, $\T$ is the $2$-torus $\T$, and we have the following result for $B_{2}(\T)$.




\begin{thm}[{\cite[Theorem~3]{BGeG}}]\label{th:renil}
The group $B_2(\T)$ is residually nilpotent, but is not residually torsion-free nilpotent.  Further, $\Gamma_2(B_2(\T))/\Gamma_{3}(B_2(\T))\cong\Z^3_2$, and $\Gamma_3(B_2(\T))/\Gamma_{4}(B_2(\T))\cong\Z_{2}^{5}$.
\end{thm}

\begin{proof}
The first part of the statement is~\cite[Theorem~3(a) and~(c)]{BGeG}. To prove the second part, using ideas from~\cite{GG1}, it was shown in~\cite[Theorem~3(b)]{BGeG} that with the exception of the first term, the lower central series of $B_2(\T)$ and the free product $\Z_2 \ast \Z_2 \ast \Z_2$ coincide. With the help of results of~\cite{Ga}, for all $i\geq 2$, it follows that lower central series quotient $\Gamma_i(B_2(\T))/\Gamma_{i+1}(B_2(\T))$ is isomorphic to the direct sum of $R_i$ copies of $\Z_2$, where $R_{i}$ is given by an explicit formula involving the M\"obius function, from which one may check that $R_{2}=3$ and $R_{3}=5$. This yields the second part of the statement. 
\end{proof}

If $g>1$, $B_{2}(\Sigma_{g})$ is residually nilpotent.

\begin{prop}[{\cite[Corollary~10]{BB}}] \label{prop:n=2orientable}
If $g\ge 1$, the group $B_2(\Sigma_g)$ is residually $2$-finite. In particular, it is residually nilpotent. 
\end{prop}

\begin{rem}\label{rem:notnilp}
To prove some of our results, we will need to be sure that our residually nilpotent groups are not nilpotent, in particular that all of their lower central series quotients are non trivial. We claim that this is the case for the group $B_2(\Sigma_g)$ for all $g\geq 1$. If $g=1$, the result follows from~\cite[Theorem~3]{BGeG} (note that the group $\Z_2 \ast \Z_2 \ast \Z_2$ contains a subgroup that is a free group of rank $2$).
So assume that $g>1$, and suppose on the contrary that there exists $i\in \N$ such that $\Gamma_i(B_2(\Sigma_g))=\brak{1}$. Without loss of generality, we may suppose that $i$ is minimal with respect to this property. Since $B_2(\Sigma_g)$ is non Abelian, it follows from \reth{gam3closed}\ref{it:gam12g} that $i\geq 3$. Now $\Gamma_{i}(B_2(\Sigma_g))= [\Gamma_{i-1}(B_2(\Sigma_g)), B_2(\Sigma_g)]=\brak{1}$, and hence $\Gamma_{i-1}(B_2(\Sigma_g))$ is contained in the centre of $B_2(\Sigma_g)$. This centre is trivial~\cite{GGmpcps,PR}, so $\Gamma_{i-1}(B_2(\Sigma_g))=\brak{1}$, but this contradicts the minimality of $i$, and so proves the result in this case.
\end{rem}


The computation of the lower central series quotients in the case $n=2$ and $g>1$, namely the generalisation of \reth{renil} and~\cite[Theorem~3]{BGeG} to surfaces of arbitrary genus, remains an open problem. The following result nevertheless gives some information about the quotient $\Gamma_2(B_2(\Sigma_g))/\Gamma_3(B_2(\Sigma_g))$.

\begin{prop}\label{prop:pre}
If $g\geq 1$,
the group $\Gamma_2(B_2(\Sigma_g))/\Gamma_3(B_2(\Sigma_g))$ is non-trivial, and is a quotient of $\Z_2^{2g} \oplus \Z_{g+1}$.
\end{prop}

\begin{proof}
If $g=1$ then by \reth{renil}, $\Gamma_2(B_2(\T))/\Gamma_3(B_2(\T))\cong \Z_{2}^{3}$, and the result holds. So suppose that $g>1$. In what follows we will make use freely of the Witt-Hall identities~\cite[Theorem~5.1]{MKS}. By relation~(\ref{eq:cicj}), we have $1=[c_i ,\sigma_1^{-1}  c_j \sigma_1]$ for $c_i=a_i$ or $b_i$, $c_j=a_j$ or $b_j$ and for all $1\le j<i\le g$ in $B_2(\Sigma_g)$, from which it follows that $1=[c_i,  c_j]$ in $B_2(\Sigma_g)/\Gamma_3(B_2(\Sigma_g))$.   
Using relation~(\ref{eq:cici}), we have $1=[c_i, \sigma_1 c_i \sigma_1]$ in $B_2(\Sigma_g)$ for $c_i=a_i$ or $b_i$ and for all $i=1, \ldots, g$, which implies that $1=[c_i,  \sigma_1]^2$ in $B_2(\Sigma_g)/\Gamma_3(B_2(\Sigma_g))$.
Similarly, from relation~(\ref{eq:ab}), we obtain $[b_i^{-1}, \sigma_1^{-1} a_i  \sigma_1]=\sigma_1^2$ in $B_2(\Sigma_g)$ for all $i=1, \ldots, g$, and therefore $[b_i^{-1},   a_i]=[b_i,   a_i]^{-1}=\sigma_1^2$ in  $B_2(\Sigma_g)/\Gamma_3(B_2(\Sigma_g))$. Since $\prod_{i=1}^g\, [a_i^{-1},b_i]=\prod_{i=1}^g\, [b_i, a_i]$ in $B_2(\Sigma_g)/\Gamma_3(B_2(\Sigma_g))$ by relation~(\ref{eq:tot}), we see that $\sigma_1^{-2g}=\sigma_1^2$, and thus the order of $\sigma_1^2$ in $\Gamma_2(B_2(\Sigma_g))/\Gamma_3(B_2(\Sigma_g))$ divides $g+1$. These computations imply that $\Gamma_2(B_2(\Sigma_g))/\Gamma_3(B_2(\Sigma_g))$ is an Abelian group that is generated by the commutators 
$[c_i, \sigma_1]$ for $c_i=a_i$ or $b_i$ and $i=1, \ldots, g$, which are all of order at most $2$, and
the commutators $[b_i,a_i]$, where $i=1, \ldots, g$, and which are all identified to a single element $\sigma_{1}^{2}$ of order at most $g+1$. Consequently, $\Gamma_2(B_2(\Sigma_g))/\Gamma_3(B_2(\Sigma_g))$ is a quotient of $\Z_2^{2g} \oplus \Z_{g+1}$. \rerem{notnilp} implies that this quotient is non trivial, which proves the result. 
\end{proof}

\begin{prop}\label{prop:perfect}
If $g\geq 1$ and $n\in \brak{3,4}$, the group $\Gamma_3(B_n(\Sigma_{g}))$ is not perfect.
\end{prop}

\begin{proof}
Let $g\geq 1$ and $n\in \brak{3,4}$. Let $\pi_n\colon\thinspace B_n(\Sigma_{g}) \to S_n$ be the homomorphism  that arises in~(\ref{eq:permutation}), and for $i\geq 2$, let $\pi_{n,i}\colon\thinspace \Gamma_i(B_n(\Sigma_{g})) \to \Gamma_i(S_n)$ denote the induced surjective homomorphism between the corresponding terms of the lower central series. A straightforward computation shows that $\Gamma_2(S_n)=\Gamma_3(S_n)=A_{n}$, where $A_{n}$ is the alternating group, and that $\Gamma_3(S_n)/[\Gamma_3(S_n),\Gamma_3(S_n)]$ is isomorphic to $\Z_3$. Now the homomorphism $\pi_{n,3}$ induces a surjection at the level of Abelianisations, and since $\Gamma_3(S_n)/[\Gamma_3(S_n),\Gamma_3(S_n)]$ is non trivial, we conclude that $\Gamma_3(B_n(\Sigma_{g}))$ cannot be perfect.
\end{proof}

\begin{proof}[Proof of Theorem \ref{th:gam3closed}]
First assume that $n=1$. We have $B_1(\Sigma_{g})= \pi_1(\Sigma_{g})$, which is residually free, and therefore residually (torsion free) nilpotent. Further, by~\cite[Main Theorem]{LA}, $\Gamma_i(B_1(\Sigma_g))/\Gamma_{i+1}(B_1(\Sigma_g))$ is isomorphic to $\Z^{2g}$ if $i=1$, to $\Z^{g(2g-1)-1}$ if $i=2$, and to    $\Z^{4g(g^{2}-1)}$ if $i=3$. In particular $\Gamma_3(B_1(\Sigma_g))$ is not perfect. Therefore all statements of the theorem pertaining to the case $n=1$ hold.

Now suppose that $n=2$. It follows in a straightforward manner from \reth{presbng} that $\Gamma_1(B_2(\Sigma_g))/\Gamma_2(B_2(\Sigma_g)) \cong \Z^{2g} \oplus \Z_{2}$. Part~\ref{it:gam23g}\ref{it:gam23gn2} follows from \reth{renil} and \repr{pre}, and parts~\ref{it:gam3g} and~\ref{it:gam4g} in the case $n=2$ are a consequence of \reth{renil}, \repr{n=2orientable} and \rerem{notnilp}. 

Finally, let $n\ge 3$. Parts~\ref{it:gam12g},~\ref{it:gam23g}\ref{it:gam23gn},~\ref{it:gam4g}, and the sufficiency of the condition in part~\ref{it:gam3g}  were proved in~\cite[Theorem~1]{BGeG}. Part~\ref{it:gam3g} in the case  $n\in \brak{3,4}$ is a consequence of  \repr{perfect}. This completes the proof of the theorem.
\end{proof}

\subsection{Surjections between braid groups of orientable surfaces of non-zero genus without boundary}\label{sec:suror}

With the notation of~\cite{GGzeit}, if $\Gamma$ is a finitely-generated group, let $G(\Gamma)$ denote the minimal cardinality among all generating sets of $\Gamma$. By~\cite[Proposition~8]{GGzeit}, if $\Gamma'$ is another finitely-generated group such that there exists a surjective homomorphism from $\Gamma$ to $\Gamma'$ then:
\begin{equation}\label{eq:Abelian}
\text{$G(\Gamma)\geq G(\Gamma')$ and $G(\Gamma)\geq G(\Gamma_{\text{Ab}})$,}
\end{equation}
the second inequality following from the first by taking the homomorphism to be Abelianisation.
The following proposition generalises some of the principal results of~\cite{GGzeit} to the case of orientable surfaces of genus $g\geq 1$.  

\begin{prop}\label{prop:mingen}
Let $g,m\in \N$. Then $G(B_{m}(\Sigma_{g}))=\begin{cases}
2g+m-1 & \text{if $m\in \{ 1,2\}$}\\
2g+2 & \text{if $m\geq 3$.}
\end{cases}$
\end{prop}

\begin{proof}
If $m\in \{ 1,2\}$, then $(B_{m}(\Sigma_{g}))_{\text{Ab}}\cong \Z^{2g}\oplus \Z_2^{m-1}$ using \reth{gam3closed}\ref{it:gam12g}, so $G(B_{m}(\Sigma_{g}))\geq 2g+m-1$ by~\ref{eq:Abelian}. 
By taking the generating set of $B_{m}(\Sigma_{g})$ given in \reth{presbng}, we see also that $G(B_{m}(\Sigma_{g}))\leq 2g+m-1$, which proves the result in this case. So assume that $m\geq 3$. As in the proof of~\cite[Proposition~4]{GGzeit}, using \reth{gam3closed}, we see that $\{ a_{1}, b_{1},\ldots, a_{g}, b_{g}, \sigma_{1}, \sigma_{1}\cdots \sigma_{m-1} \}$ is a generating set for $B_{m}(\Sigma_{g})$, and so $G(B_{m}(\Sigma_{g}))\leq 2g+2$. Conversely, with respect to the presentation given by \reth{presbng}, let $f\colon\thinspace B_{m}(\Sigma_{g}) \to \Z^{2g}$ be the surjective homomorphism whose kernel is the normal closure of $\{ \sigma_{1},\ldots, \sigma_{m-1}\}$ in $B_{m}(\Sigma_{g})$, and let $h\colon\thinspace B_{m}(\Sigma_{g}) \to S_{m}$  be the surjective homomorphism given by equation~(\ref{eq:permutation}) whose kernel is $P_{m}(\Sigma_{g})$. Using \reth{presbng}, $h$ may also be seen to be the projection onto the quotient of $B_{m}(\Sigma_{g})$ by the normal closure of the set $\brak{a_{1}, b_{1},\ldots, a_{g},b_{g}}$. 
The map $f\times h\colon\thinspace B_{m}(\Sigma_{g}) \to \Z^{2g} \times S_{m}$ is clearly a homomorphism. To see that it is surjective, note that if $(w,\tau)\in \Z^{2g} \times S_{m}$, there exist $\alpha\in P_{m}(\Sigma_{g})$ and $\beta$ belonging to the subgroup of $B_{m}(\Sigma_{g})$ generated by $\{ \sigma_{1},\ldots, \sigma_{m-1}\}$ such that $f(\alpha)=w$ and $h(\beta)= \tau$. From the description of $f$ and $h$, we have $(f\times h)(\alpha\beta)=(w,\tau)$, which proves that $f\times h$ is surjective. So by~\ref{eq:Abelian},
$G(B_{m}(\Sigma_{g})) \geq G(\Z^{2g} \times S_{m})=2g+G(S_{m})\geq 2g+2$ because $m\geq 3$. Thus $G(B_{m}(\Sigma_{g}))=2g+2$, and the statement then follows in this case.
\end{proof}

\begin{cor}\label{cor:special}
Let $m\geq 3$, $n\in \brak{1,2}$ and $g\geq 1$. Then there is no surjective homomorphism from $B_{n}(\Sigma_{g})$ to $B_{m}(\Sigma_{g})$.
\end{cor}

\begin{proof}
If $m\geq 3$, $n\in \brak{1,2}$ and $g\geq 1$, the result follows from~\ref{eq:Abelian} using the fact that $G(B_{m}(\Sigma_{g})) > G(B_{n}(\Sigma_{g}))$ by \repr{mingen}.
\end{proof}



\begin{proof}[Proof of \reth{LinSurfcomb} in the case where $\Sigma=\Sigma_{g}$ and $g\geq 1$]
Let $n>m\ge 2$, and consider the map from $B_n$ to $B_n(\Sigma_g)$ defined on the generators of $B_{n}$ by sending $\sigma_i$ to $\sigma_i$ for all $i=1, \ldots, n-1$. It is a homomorphism (note that by~\cite{PR}, it is also an embedding). 
Suppose first that $n>m\ge 3$ and $n\neq 4$, and let $\Phi \colon\thinspace B_n(\Sigma_g) \to S_{m}$ be a homomorphism.
By~\reth{Lin}, the elements $\Phi(\sigma_i)$, where $i=1, \ldots, n-1$, are powers of a single element, and therefore commute pairwise. Using the braid relations, the fact that $\Phi(\sigma_{i})$ commutes with $\Phi(\sigma_{i+1})$ for all $i=1, \ldots, n-2$ implies that $\Phi(\sigma_1)= \cdots= \Phi(\sigma_{n-1})$. We  denote this common element by $\sigma$. We see from relations~(\ref{eq:cs}) that $\sigma$ commutes with $\Phi(a_j)$ and $\Phi(b_j)$ for all $j=1, \ldots, g$. Suppose now that $\Phi$ is surjective. Then $\sigma$ belongs to the centre of $S_m$, which is trivial since $m\geq 3$, so $\sigma$ is trivial. Therefore the homomorphism $\Phi$ factors through the surjective homomorphism $\Phi'\colon\thinspace B_n(\Sigma_g)/\langle\!\langle \sigma_1\rangle\!\rangle \to S_m$, where $\langle\!\langle \sigma_1\rangle\!\rangle$ denotes the normal closure of $\sigma_{1}$ in $B_n(\Sigma_g)$. But using \reth{presbng} (\emph{cf.} the proof of \repr{mingen}), $B_n(\Sigma_g)/\langle\!\langle \sigma_1\rangle\!\rangle$ is isomorphic to $\Z^{2g}$, so is Abelian, while $S_{m}$ is not. This yields a contradiction, and hence $\Phi$ is not surjective.

Conversely, if $(n,m)=(4,3)$, the map from $B_4(\Sigma_g)$ to $S_3$ defined by sending the elements $a_1,b_1,\ldots, a_g, b_g$ to the identity element, $\sigma_1$ and $\sigma_3$ to $(1,2)$, and $\sigma_2$ to $(2,3)$, extends to a well-defined, surjective homomorphism by \reth{presbng}.  
\end{proof}

We are now able to prove \reth{gensurj} for the braid groups of orientable surfaces without boundary of genus $g\geq 1$.


\begin{proof}[Proof of \reth{gensurj}\ref{it:main1a}\ref{it:main1aii}]
Suppose first that $m=g=1$, and that $n\geq 2$. Since $B_{1}(\T)\cong \Z^{2}$, the result follows by considering the surjective homomorphism $f\colon\thinspace B_{n}(\T) \to \Z^{2}$ defined in the proof of \repr{mingen}. 
To prove the converse, we will show that if $(g,m)\neq (1,1)$, there is no surjective homomorphism from $B_{n}(\Sigma_{g})$ to $B_{m}(\Sigma_{g})$. We split the proof into the following three cases.
\begin{enumerate}[label=\textit{(\arabic*)}]
\item $n<m$. If $n\in \brak{1,2}$, the result follows from \reco{special}.
So suppose that $n\geq 3$. \reth{gam3closed}\ref{it:gam23g}\ref{it:gam23gn} implies that there is no surjective homomorphism from $\Gamma_2(B_n(\Sigma_g))/\Gamma_3(B_n(\Sigma_g))$ onto $\Gamma_2(B_m(\Sigma_g))/\Gamma_3(B_m(\Sigma_g))$, and hence there is no surjective homomorphism from $B_{n}(\Sigma_{g})$ onto $B_{m}(\Sigma_{g})$.

%
\item\label{it:ngrm} $n>m$, where either $g>1$ and $m\in \brak{1,2}$, or $g=1$ and $m=2$.
If $n\geq 3$, by \reth{gam3closed}\ref{it:gam3g}, $\Gamma_3(B_{n}(\Sigma_{g}))/\Gamma_{4}(B_{n}(\Sigma_{g}))$ is trivial, while $\Gamma_3(B_{m}(\Sigma_{g}))/\Gamma_{4}(B_{m}(\Sigma_{g}))$ is not, and this implies that there is no surjective homomorphism from $B_{n}(\Sigma_{g})$ onto $B_{m}(\Sigma_{g})$.
If $n=2$, $m=1$ and $g>1$, $\Gamma_2(B_{2}(\Sigma_{g}))/\Gamma_{3}(B_{2}(\Sigma_{g}))$ is finite by \reth{gam3closed}\ref{it:gam23g}\ref{it:gam23gn2}, and so it cannot surject onto $\Gamma_2(\pi_1(\Sigma_g))/\Gamma_3(\pi_1(\Sigma_g))$, which is a (non-trivial) free Abelian group by \reth{gam3closed}\ref{it:gam23g}\ref{it:gam23gn}.

\item\label{it:ngrmc} $n>m \geq 3$. Assume first that $n\neq 4$. There can be no surjection homomorphism  from $B_n(\Sigma_g)$ onto $B_{m}(\Sigma_{g})$, for otherwise its composition with the projection $B_{m}(\Sigma_{g})$ onto $S_{m}$ of~(\ref{eq:permutation}) would yield a surjective homomorphism from $B_n(\Sigma_g)$ onto $S_{m}$, which contradicts \reth{LinSurfcomb}. 
So assume that $n=4$. Then $m=3$, and there can be no surjective homomorphism from $B_4(\Sigma_g)$ to $B_3(\Sigma_g)$ because otherwise by \reth{gam3closed}\ref{it:gam23g}\ref{it:gam23gn}, there would be a surjective homomorphism from $\Gamma_2(B_4(\Sigma_g))/\Gamma_3(B_4(\Sigma_g))$, which is isomorphic to $\Z_{3+g}$, onto $\Gamma_2(B_3(\Sigma_g))/\Gamma_3(B_3(\Sigma_g))$, which is isomorphic to $\Z_{2+g}$, but this is impossible.\qedhere
\end{enumerate}
\end{proof}

\begin{cor}\label{cor:pure} 
Let $g\geq 1$, and let $n,m\in \N$. There is a surjective homomorphism of $B_n(\Sigma_g)$ onto $P_m(\Sigma_g)$ if and only if $n=m=1$ for $g \ge1$ and $m=1$ for $g=1$.
\end{cor}

\begin{proof}
Let $g\geq 1$. We first prove that the conditions are sufficient. If $n=m=1$, the result is clear since the given groups coincide with the fundamental group of the surface. 
If $g=m=1$ then the result follows from \reth{gensurj}\ref{it:main1a}\ref{it:main1aii}. Conversely, suppose that there exists a surjective homomorphism $\Phi\colon\thinspace B_n(\Sigma_g)\to P_m(\Sigma_g)$. Then $\Phi$ induces a surjective homomorphism of the corresponding Abelianisations, but since $(P_m(\Sigma_g))_{\text{Ab}}\cong \Z^{2gm}$ from the presentation of $P_m(\Sigma_g)$ given in~\cite{B} for instance, it follows from \reth{gam3closed}\ref{it:gam12g}
that $m=1$. Then either $n=1$, or $n>1$, in which case $g=1$ by \reth{gensurj}\ref{it:main1a}\ref{it:main1aii}, and in both cases, the conclusion holds.
\end{proof}

\begin{rem}
With the exception of the case $g=1$, \reco{pure} was proved in~\cite[Theorem~1.2]{Ch} using different methods. 
\end{rem}

\subsection{Surjections between braid groups of the sphere}\label{sec:sphere}

In this section, we complete the analysis of surjections between braid groups of orientable surfaces withour boundary by studying the case $g=0$, which is that of the sphere $\SP$. Theorems~\ref{th:gensurj} and~\ref{th:LinSurfcomb} hold also in this case, but the arguments are somewhat different. As we mentioned just after the statement of \reth{presbng}, if $n\in \N$, the presentation of $B_{n}(\SP)$ in~\cite{FV} may be obtained from the standard presentation of $B_{n}$ by adding the relation~(\ref{eq:totsph}), so $B_{n}(\SP)$ is a quotient of $B_{n}$. It follows from this presentation that $B_{1}(\SP)$ is trivial, $B_2(\SP)=\Z_2$, $B_3(\SP)=\Z_3 \rtimes \Z_4$ (with non-trivial action), and $B_n(\SP)$ is an infinite group for all $n\ge 4$~\cite[third theorem, p.~255]{FV}.
%
%
The following result summarises some known results about the lower central series of the braid groups of the sphere.

\begin{thm}[\cite{GGsph}]\label{th:gam3sph}\mbox{}
\begin{enumerate}
\item\label{it:gam1sp}  $\Gamma_1(B_n(\SP))/\Gamma_2(B_n(\SP)) \cong \Z_{2(n-1)}$ for $n\ge 2$. 
\item\label{it:gam2sp}  $\Gamma_2(B_n(\SP)))=\Gamma_3(B_n(\SP))$ for $n\ge 2$. 
\item\label{it:gam3sp}  $\Gamma_2(B_n(\SP)$ is perfect if and only if $n\ge 5$.
\end{enumerate}
\end{thm}


The proofs of parts~\ref{it:gam1sp},~\ref{it:gam2sp} and~\ref{it:gam3sp} of \reth{gam3sph} may be found in Proposition~2.1, and Theorems~1.3 and~1.4 respectively of \cite{GGsph}. We now prove \reth{LinSurfcomb} in the case of the sphere and \reth{gensurj}\ref{it:main1a}\ref{it:main1ai}.



\begin{proof}[Proof of \reth{LinSurfcomb} in the case where $\Sigma=\Sigma_{0}$]
Let $n>m\ge 3$ and $n\neq 4$, and suppose that there exists a surjective homomorphism $\Phi\colon\thinspace B_n(\SP) \to S_m$. Since $B_n(\SP)$ is a quotient of $B_n$, $B_{n}$ surjects homomorphically onto $B_n(\SP)$, so its composition with $\Phi$ would gives rise to a surjective homomorphism from $B_n$ to $S_m$, which contradicts \reth{Lin}.
\end{proof}


\begin{proof}[Proof of \reth{gensurj}\ref{it:main1a}\ref{it:main1ai}]
We start by showing that the condition for the existence of a surjective homomorphism from $B_{n}(\SP)$ to $B_{m}(\SP)$ is sufficient. Since $B_{1}(\SP)$ is trivial, the result is clear if $m=1$, and if $n\geq 3$ and $m=2$, $B_n(\SP)$ surjects homomorphically onto $B_2(\SP)$ since $(B_{n}(\SP))_{\text{Ab}}\cong \Z_{2(n-1)}$ and $B_2(\SP)\cong \Z_2$, so there exists a surjective homomorphism from $B_{n}(\SP)$ to $B_{2}(\SP)$ that factors through $(B_{n}(\SP))_{\text{Ab}}$. To show that the condition is necessary, we consider the following two cases.
\begin{enumerate}[label=\textit{(\arabic*)}]
\item  Suppose that $n<m$. Since $B_{1}(\SP)$ is trivial and $(B_{n}(\SP))_{\text{Ab}}\cong \Z_{2(n-1)}$ for all $n\geq 2$, there does not exist a surjective homomorphism between $(B_{n}(\SP))_{\text{Ab}}$ and $(B_{m}(\SP))_{\text{Ab}}$, so there cannot exist a surjective homomorphism between $B_{n}(\SP)$ and $B_{m}(\SP)$.

\item Now let $n>m\ge 3$. If $n\neq 4$, the fact that there does not exist a surjective homomorphism from $B_n(\Sigma_g)$ onto $S_m$ by \reth{LinSurfcomb} implies that there does not exist a surjective homomorphism from $B_n(\SP)$ onto $B_m(\SP)$.
The remaining case, $(n,m)=(4,3)$, may be dealt with by studying the finite subgroups of the braid groups of $\SP$ as follows. Let $\Phi\colon\thinspace B_{4}(\SP) \to B_{3}(\SP)$ be a homomorphism, and using the notation of \reth{presbng} in $B_{4}(\SP)$, let $\alpha_{0}=\sigma_{1}\sigma_{2}\sigma_{3}$ and $\alpha_{1}= \sigma_{1}\sigma_{2} \sigma_{3}^{2}$, and let $\Delta_{4}=\sigma_{1}\sigma_{2} \sigma_{3}\sigma_{1} \sigma_{2}\sigma_{1}$ be the half-twist braid. By~\cite[Theorem~3]{GGjktr}, $B_{4}(\SP)=\langle \alpha_{0}, \alpha_{1} \rangle$. Now $\alpha_{0}$ is of order $8$, and the maximal torsion of $B_{3}(\SP)$ is equal to $6$, so the order of $\Phi(\alpha_{0})$ is a divisor of $4$~\cite{GVB,M}. But the full-twist braid $\Delta_{4}^{2}$ is the unique element of $B_{4}(\SP)$ of order $2$~\cite{GVB}. This implies that $\Delta_{4}^{2}$ belongs to the centre of $B_{4}(\SP)$, and also that $\alpha_{0}^{4}=\Delta_{4}^{2}$, from which we conclude that $\Delta_{4}^{2}$ belongs to $\ker{\Phi}$. 
Let $H=\langle \alpha_{0}, \Delta_{4} \rangle$. By~\cite[Remark, p.~234]{GGblms}, $H$ is isomorphic to the generalised quaternion group $\mathcal{Q}_{16}$ of order $16$, where the relations are of the form $\alpha_{0}^{4}=\Delta_{4}^{2}$ and $\Delta_{4}\alpha_{0} \Delta_{4}^{-1}=\alpha_{0}^{-1}$. Consider the restriction $\Phi\vert_{H}\colon\thinspace H \to \im{\Phi\vert_{H}}$. Since $\Delta_{4}^{2}$ belongs to $H\cap \ker{\Phi\vert_{H}}$ and to the centre of $B_{4}(\SP)$, we see that $\Phi\vert_{H}$ factors through the quotient $H/\langle \Delta_{4}^{2}\rangle$. Using the relations of $H$ in terms of its generators, this quotient is isomorphic to the dihedral group of order $8$, and hence $\im{\Phi\vert_{H}}$ is a subgroup of $B_{3}(\SP)$ that is a quotient of this dihedral group. 
On the other hand, the quotients of dihedral groups are either dihedral, the trivial group, or cyclic of order $2$. Further, $B_{3}(\SP)\cong \Z_{3}\rtimes \Z_{4}$, the action being the non-trivial one~\cite{FV}, so $B_{3}(\SP)$ has no dihedral subgroups. We conclude that $\im{\Phi\vert_{H}}\subset \langle \Delta_{3}^{2} \rangle$. Hence $\ker{\Phi\vert_{H}}$ is either equal to $H$, or is a subgroup of $H$ of index $2$. If $\ker{\Phi\vert_{H}}$ is of index $2$ in $H$, then by analysing the images of $\alpha_{0}$ and $\Delta_{4}$ by a surjective homomorphism from $H$ to $\Z_{2}$, we see that $\ker{\Phi\vert_{H}}$ is equal to $\ang{\alpha_{0}}$, to $\ang{\alpha_{0}^{2}, \Delta_{4}}$, or to $\ang{\alpha_{0}^{2}, \alpha_{0}\Delta_{4}}$. So if either $\ker{\Phi\vert_{H}}$ is equal to $H$, or is a subgroup of $H$ of index $2$, we conclude from these possibilities that $\alpha_{0}^{2} \in \ker{\Phi}$. It follows again from the fact that $\Delta_{3}^{2}$ is the unique element of $B_{3}(\SP)$ of order $2$ that $\Phi(\alpha_{0})\in \ang{\Delta_{3}^{2}}$, and so is central in $B_{3}(\SP)$. Since $B_{4}(\SP)=\langle \alpha_{0}, \alpha_{1} \rangle$, we conclude that $\im{\Phi}$ is cyclic, and hence $\Phi$ cannot be surjective.\qedhere
\end{enumerate}
\end{proof}
%

\begin{rem}
It follows from \reth{gensurj}\ref{it:main1a}\ref{it:main1ai} that there is no surjective homomorphism from $ B_4(\SP)$ to $B_3(\SP)$. However, the maps from $B_4(\SP)$ to $S_3$ defined by sending the generators $\sigma_1$ and $\sigma_3$ to $(1,2)$ and $\sigma_2$ to $(2,3)$ and from $B_{4}$ to $B_{3}$ defined by sending the generators $\sigma_1$ and $\sigma_3$ to $\sigma_{1}$ and $\sigma_2$ to $\sigma_{2}$, extend to well-defined, surjective homomorphisms. 
\end{rem}


\section{Surjections between braid groups of orientable surfaces with boundary}\label{sec:orboundary}

Let $\Sigma_{g,b}$ be a compact, connected orientable surface of genus $g$ with $b\geq 0$ boundary
components. A presentation for $B_n(\Sigma_{g,b})$ may be found in~\cite[Proposition~3.1]{BGoG}, and in the case $b=1$, a presentation for $B_n(\Sigma_{g,1})$ may be obtained from that of $B_n(\Sigma_{g})$ given in \reth{presbng} by deleting relation~(\ref{eq:tot}). The case $b=0$ was dealt with in \resec{closed}, so we shall assume henceforth that $b\geq 1$. The following two results generalise those of \reth{gam3closed} to the braid groups of $\Sigma_{g,b}$.

\begin{thm}[{\cite[Theorem~2]{BGeG}}]\label{th:gam3open}
Let $g,b\geq 1$, and let $n\geq 3$. 
Then:
\begin{enumerate}
\item\label{it:g1sigma} $\Gamma_1(B_n(\Sigma_{g,b}))/ \Gamma_2(B_n(\Sigma_{g,b})) \cong\Z^{2g+b-1} \oplus \Z_2$.
\item\label{it:g2sigma} $\Gamma_2(B_n(\Sigma_{g,b}))/ \Gamma_3(B_n(\Sigma_{g,b}))\cong \Z$.
\item\label{it:g3sigma} $\Gamma_3(B_n(\Sigma_{g,b}))= \Gamma_4(B_n(\Sigma_{g,b}))$. Moreover
$\Gamma_3(B_n(\Sigma_{g,b}))$ is perfect for $n\ge 5$.
\item\label{it:resid} $B_n(\Sigma_{g,b})$ is not residually nilpotent.
\end{enumerate}
\end{thm}

The following proposition treats the case $n=2$.
 
\begin{prop} \label{prop:n=2orientablebound} 
Let $g,b\ge 1$. 
\begin{enumerate}
\item\label{it:resboundary1} The group $B_2(\Sigma_{g,b})$ is residually $2$-finite and therefore residually nilpotent, but is not nilpotent.
\item\label{it:resboundary12} $\Gamma_1(B_2(\Sigma_{g,b}))/ \Gamma_2(B_2(\Sigma_{g,b})) \cong\Z^{2g+b-1} \oplus \Z_2$.
\item\label{it:resboundary2} The group $\Gamma_2(B_2(\Sigma_{g,b}))/\Gamma_3(B_2(\Sigma_{g,b}))$ is a non-trivial quotient of $\Z_2^{2g+b-1} \oplus \Z$.
\end{enumerate}
\end{prop}

Presentations for $B_n(\Sigma_{g,b})/ \Gamma_3(B_n(\Sigma_{g,b}))$ were exhibited in~\cite[eq.~(10)]{BGeG} for $b=1$, and in~\cite[Proposition~3.13]{BGoG} for $b\ge 1$.
 
\begin{proof}[Proof of \repr{n=2orientablebound}]
Let $n=2$, and consider the short exact sequence~(\ref{eq:permutation}), where we take $\Sigma= \Sigma_{g,b}$. 
Since $S_2\cong \Z_2$ and $P_2(\Sigma_{g,b})$ is residually torsion free nilpotent~\cite[Theorem~4]{BGeG}, and therefore $2$-finite, the hypotheses of~\cite[Lemma~1.5]{Gr} are fulfilled, so $B_2(\Sigma_{g,b})$ is residually nilpotent. To see that it is not nilpotent, suppose on the contrary that there exists $i\in \N$ such that $\Gamma_i(B_2(\Sigma_{g,b}))=\brak{1}$. Without loss of generality, we may suppose that $i$ is minimal with respect to this property. Since $B_2(\Sigma_{g,b})$ is non Abelian, it follows from \reth{gam3open}\ref{it:g1sigma} that $i\geq 3$. Now $\Gamma_{i}(B_2(\Sigma_{g,b}))= [\Gamma_{i-1}(B_2(\Sigma_{g,b})), B_2(\Sigma_{g,b})]=\brak{1}$, and hence $\Gamma_{i-1}(B_2(\Sigma_{g,b}))$ is contained in the centre of $B_2(\Sigma_{g,b})$. This centre is trivial~\cite{GGmpcps,PR}, so $\Gamma_{i-1}(B_2(\Sigma_{g,b}))=\brak{1}$, but this contradicts the minimality of $i$. Part~\ref{it:resboundary1} follows. 

For part~\ref{it:resboundary12}, we just give the proof in the case $b=1$. The general case may be obtained in a similar manner using the presentation of $B_2(\Sigma_{g,b})$ given in~\cite{BGoG}. As we mentioned above, a presentation of $B_2(\Sigma_{g,1})$ may be obtained by deleting relation~(\ref{eq:tot}) from the presentation of \reth{presbng}. Thus the proof given in \repr{pre} for $\Sigma_{g}$ is also valid in the case of $\Sigma_{g,b}$, except that we can no longer conclude that $\sigma_{1}^{2}$ is of finite order, so the second factor in the direct product decomposition of  $B_2(\Sigma_{g,1})/\Gamma_3(B_2(\Sigma_{g,1}))$ is $\Z$. Part~\ref{it:resboundary2} is a consequence of part~\ref{it:resboundary1}.
%
\end{proof}
 
\begin{rem}\label{rem:n1bound}
\reth{gam3open} (in the case $n\geq 3$) and \repr{n=2orientablebound} (in the case $n=2$)
generalise \reth{gam3closed}.  If $n=1$, $B_1(\Sigma_{g,b})$ is a free group of rank $2g+b-1$, and 
its lower central series is well known, see~\cite{LA} for instance. Note that in particular $B_1(\Sigma_{g,b})$ is residually nilpotent. It follows from \reth{gam3open}\ref{it:resid} and \repr{n=2orientablebound}\ref{it:resboundary1} that $B_n(\Sigma_{g,b})$ is residually nilpotent if and only if $n\le 2$. As in the proof of \repr{perfect}, we see that $\Gamma_3(B_n(\Sigma_{g,b}))$ is not perfect if $n\in \brak{3,4}$.
Hence using  \reth{gam3open}\ref{it:g3sigma}$,\Gamma_3(B_n(\Sigma_{g,b}))$ is perfect if and only if $n\ge 5$. 
\end{rem}

We now prove \reth{gensurj}\ref{it:main1b} in the orientable case and \reco{pureboundary}. We first require the following result.

\begin{lem}\label{lem:gb1} 
There is no surjective homomorphism from $B_2(\Sigma_{1,1})$ onto $\pi_1(\Sigma_{1,1})$.
\end{lem}

\begin{proof}
Suppose on the contrary that there exists a surjective homomorphism $\phi\colon\thinspace B_2(\Sigma_{1,1}) \to \pi_1(\Sigma_{1,1})$. Let $\alpha= a_1 \sigma_1$, $\beta = b_1 \sigma_1$. Then $\alpha, \beta, \sigma_1$  generate $B_2(\Sigma_{1,1})$, and the defining relations of \reth{presbng} become:
\begin{align}
\alpha^2 &= \sigma_1 \alpha^2  \sigma_1^{-1}  \label{eq:alpha}\\
\beta^2 &= \sigma_1 \beta^2  \sigma_1^{-1} \label{eq:beta}\\
\alpha \beta  \sigma_1^{-1} &= \sigma_1 \beta\alpha.\label{eq:abs}
\end{align}
Now $\pi_1(\Sigma_{1,1})$ is a free group of rank $2$, and so if $u,v \in \pi_1(\Sigma_{1,1})$, the relation $u^2=v^2$ implies that $u=v$. Applying this to relations~(\ref{eq:alpha}) and~(\ref{eq:beta}), we deduce that $\phi(\sigma_1)$ is central in $\pi_1(\Sigma_{1,1})$, and so $\phi(\sigma_1)=1$. Since $\phi$ is surjective, it follows that $\pi_1(\Sigma_{1,1})=\ang{\phi(\alpha), \phi(\beta)}$.  Relation~(\ref{eq:abs}) implies that $\phi(\alpha)$ and $\phi(\beta)$ commute. Consequently, $\ang{\phi(\alpha), \phi(\beta)}$ is cyclic, and this contradicts the assumption that $\phi$ is surjective.
\end{proof}

In contrast with the case of $\Sigma_{g}$, \reth{gam3open} implies that the lower central series does not distinguish the number of strings for braid groups of orientable surfaces with boundary if $n\geq 3$. Nevertheless, we are able to show that in certain cases, there does not exist a surjective homomorphism between $B_n(\Sigma_{g,b})$ and $B_m(\Sigma_{g,b})$.



\begin{proof}[Proof of \reth{gensurj}\ref{it:main1b} in the orientable case]\mbox{}
We consider in turn the three cases given in the statement. 
\begin{enumerate}[label=\textit{(\roman*)}]
\item Let $n<m$ and $n\in \brak{1,2}$. The arguments used in \repr{mingen} apply verbatim to the case with boundary. In particular $G(B_{1}(\Sigma_{g,b}))=2g+b-1$, $G(B_{2}(\Sigma_{g,b}))= 2g+b$, and $G(B_{m}(\Sigma_{g,b}))= 2g+b+1$ for all $m\ge 3$. It follows that there does not exist a surjective homomorphism in this case. 

\item Suppose that $n>m$ and $m\in \brak{1,2}$. First let $n\geq 3$. Then $B_m(\Sigma_{g,b})$ is residually nilpotent by \repr{n=2orientablebound} and \rerem{n1bound}. Since $B_n(\Sigma_{g,b})$ is not residually nilpotent by \reth{gam3open}\ref{it:resid}, it cannot surject homomorphically onto $B_m(\Sigma_{g,b})$. So suppose that $n=2$ and $m=1$. The case $(g,b) =(1,1)$ was dealt with in \relem{gb1}, so we may assume that $(g,b)\neq (1,1)$, in which case $2g+b\geq 4$. By \repr{n=2orientablebound}\ref{it:resboundary2}, the Abelian group $\Gamma_2(B_2(\Sigma_{g,b}))/\Gamma_3(B_2(\Sigma_{g,b}))$ is of rank at most $1$. On the other hand, $\Gamma_2(\pi_1(\Sigma_{g,b}))/\Gamma_3(\pi_1(\Sigma_{g,b}))$ is free Abelian of rank $(2g+b-1)(2g+b-2)/2$~\cite{LA}, and this rank is strictly greater than $1$. Thus $B_2(\Sigma_{g,b})$ cannot surject homomorphically onto $B_1(\Sigma_{g,b})$.

\item Suppose that $n>m\ge3$ and $n\neq 4$. Using the presentation of $B_m(\Sigma_{g,b})$ given in~\cite[Proposition~3.1]{BGoG}, the proof of \reth{LinSurfcomb} goes through in this case, the only difference being that $B_n(\Sigma_{g,b})/\langle\!\langle \sigma_1\rangle\!\rangle$ is isomorphic to $\Z^{2g+b-1}$. The result then follows by an argument similar to that given in case~\ref{it:ngrmc} of the proof of \reth{gensurj}\ref{it:main1a}\ref{it:main1aii} in \resec{suror}.\qedhere
\end{enumerate} 
\end{proof}

The following result is the analogue of \reco{pure} in the case where the surface has boundary.

\begin{cor}\label{cor:pureboundary} 
Let $g\geq 1$, and let $n,m\in \N$. Then there exists a surjective homomorphism of $B_n(\Sigma_{g,b})$ onto $P_m(\Sigma_{g,b})$ if and only if $n=m=1$.
\end{cor}

\begin{proof}
Let $g\geq 1$. The proof is similar to that of \reco{pure}. If $n=m=1$, the result is clear, so suppose that there exists a surjective homomorphism $\Phi\colon\thinspace B_n(\Sigma_{g,b})\to P_m(\Sigma_{g,b})$, where $(n,m)\neq (1,1)$. Then $\Phi$ induces a surjective homomorphism of the corresponding Abelianisations, but since $(P_m(\Sigma_{g,b}))_{\text{Ab}}$ is isomorphic to $\Z^{(2g+b-1)m}$ using a presentation of $P_m(\Sigma_{g,b})$ (see~\cite{B} for instance), it follows from \reth{gam3open}\ref{it:g1sigma}, \repr{n=2orientablebound}\ref{it:resboundary12} and the fact that $(B_{1}(\Sigma_{g,b}))_{\text{Ab}}$ is isomorphic to $\Z^{2g+b-1}$ that $m=1$. By \reth{gensurj}\ref{it:main1b}\ref{it:main1bii}, we conclude that $n=1$.
\end{proof}

\section{Surjections between braid groups of non-orientable surfaces}\label{sec:nonor}

We start this section by recalling a presentation of the braid groups of compact, non-orientable surfaces without boundary.

\begin{thm}[\cite{B}]\label{th:presnor}
Let $g\geq 1$, let $n\geq 2$, and let $U_{g}$ be a compact, connected non-orientable surface without boundary of genus $g$. Then $B_n(U_g)$ admits the following group presentation:
\begin{itemize}
\item \textbf{generators:} $\rho_1,  \ldots, \rho_g, \sigma_1, \ldots, \sigma_{n-1}$.
\item \textbf{relations:}
\begin{gather}
\text{$\sigma_i\sigma_j=\sigma_j\sigma_i$ if $\lvert i-j \rvert \geq
2$}\label{eq:artin1or}\\
\text{$\sigma_i\sigma_{i+1}\sigma_i= \sigma_{i+1}\sigma_i \sigma_{i+1}$
for all $1\leq i\leq n-2$}\label{eq:artin2or}\\
\text{$\rho_i\sigma_j= \sigma_j \rho_i$ for all $j\geq 2$ and $i=1, \ldots, g$}\label{eq:rs1}\\
\text{$\rho_i \sigma_1 \rho_i \sigma_1= \sigma_1^{-1} \rho_i \sigma_1 \rho_i$   for
  $i=1, \ldots, g$}\label{eq:rr}\\
\text{$\rho_r \sigma_1^{-1}  \rho_s \sigma_1=\sigma_1^{-1} \rho_s \sigma_1 \rho_r$
  for $1\le s<r \le g$}\label{eq:rs2}\\
\text{$\prod_{i=1}^g \rho_i^{-2}= \sigma_1\cdots \sigma_{n-2} \sigma_{n-1}^2 \sigma_{n-2} \cdots \sigma_1$}\label{eq:totnor}.
\end{gather}
\end{itemize}
\end{thm}
 
The presentation of $B_n(U_g)$ of~\cite[Theorem~A.3]{B} is slightly different from that given in \reth{presnor}, but one can obtain the first presentation from the second by replacing each generator $a_i$ in~\cite[Theorem~A.3]{B} by $\rho_i^{-1}$ in \reth{presnor} for all $i=1, \ldots, g$.


\begin{rem}
Notice that~\cite[Theorem~A.3]{B} was stated for $g>1$, but the presentation is also valid if $g=1$, in which case the relation~(\ref{eq:rs2}) does not exist.
This may be seen by showing that the map from $B_{n}(U_{1})$ to itself that sends the generator $\sigma_i$ (resp.\ $\rho_1$) of~\cite[Theorem~A.3]{B} to the generator $\sigma_i$ (resp.\ $\rho_1^{-1}$) of~\cite{VB} for all $1\leq i\leq n-1$ is well defined, and that it is an isomorphism. The presentation also holds if $n=1$. In particular, $B_{1}(U_{g})$ is a one-relator group, and the results of~\cite{LA} apply.
\end{rem}

The following theorem summarises some of the known results about the lower central series of braid groups of non-orientable surfaces without boundary~\cite{GG,GP,MP}, and is the analogue of Theorems~\ref{th:gam3closed} and~\ref{th:gam3open}. One may consult~\cite{BGe} for the case of pure braid groups.

\begin{thm}[\cite{GG,GP}]\label{th:gam3nor}
Let $g\geq 1$. 
Then:
\begin{enumerate}
\item\label{it:g1sigmaor} $\Gamma_1(B_n(U_{g}))/ \Gamma_2(B_n(U_{g}))=\Z^{g-1} \oplus \Z_2 \oplus \Z_2$ for all $n\geq 2$.
\item\label{it:g2sigmaor} $\Gamma_2(B_n(U_{g}))= \Gamma_3(B_n(U_{g}))$ for all $n\geq 3$. 
\item\label{it:g1sigmaorc} $\Gamma_2(B_n(U_{g}))$ is perfect if and only if  $n\geq 5$.
\item\label{it:residor} $B_n(U_{g})$ is residually nilpotent if and only if $n\leq 2$.
\end{enumerate}
\end{thm}

\begin{proof}
If $g=1$, the four statements were proved in~\cite[Theorem~1 and Proposition~6]{GG}. So in the rest of the proof, we assume that $g\geq 2$.
\begin{enumerate}
\item The statement follows in a straightforward manner using the presentation of \reth{presnor}. 

\item  If $n\geq 3$, the fact that $\Gamma_2(B_n(U_{g}))= \Gamma_3(B_n(U_{g}))$ is a consequence of the proof of~\cite[Proposition~5.21]{GP} (resp.\ of~\cite[Theorem~6.1]{GP}) if $g=2$ (resp.\ if $g\geq 3$).

\item The  `if' part is a consequence of~\cite[Theorem~1.4]{GP}. 
The proof of the `only if' part is similar to that of \repr{perfect}, and is left to the reader.

\item This follows from~\cite[Theorem~1.4]{GP}.\qedhere
\end{enumerate}
\end{proof}



\begin{proof}[Proof of \reth{LinSurfcomb} in the case where $\Sigma=U_{g}$, and $g\geq 1$]
Let $n>m\ge 3$, where $n\neq 4$. Let $\Phi\colon\thinspace B_n(U_g) \to S_{m}$. As in the proof of the orientable case, we see that $\Phi(\sigma_1)= \cdots= \Phi(\sigma_{n-1})$. We denote this common element by $\sigma$. From relations~(\ref{eq:rs1}), $\sigma$ commutes with $\Phi(\rho_j)$ for all $j=1, \ldots, g$, so $\sigma$ belongs to the centre of $S_m$, and we conclude once more that $\sigma$ is trivial, and hence the homomorphism $\Phi$ factors through a surjective homomorphism $\Phi'\colon\thinspace B_n(U_g)/\langle\!\langle \sigma_1\rangle\!\rangle \to S_m$, where $\langle\!\langle \sigma_1\rangle\!\rangle$ denotes the normal closure of $\sigma_{1}$ in $B_n(U_g)$.
But  $B_n(U_g)/\langle\!\langle \sigma_1\rangle\!\rangle$ is Abelian by relations~(\ref{eq:rs2}), which yields a contradiction because an Abelian group cannot surject homomorphically onto a non-Abelian group. So there is no surjective homomorphism from $B_{n}(U_{g})$ onto $S_{m}$.
\end{proof}

As in the case of orientable surfaces, we may obtain more information about the lower central series of $B_2(U_g)$.

\begin{prop} \label{prop:b2nor}
Let $g\ge1$. Then the group $B_2(U_g)$ is residually $2$-finite, and so is residually nilpotent.
Moreover, the group $\Gamma_2(B_2(U_g))/\Gamma_3(B_2(U_g))$ is a non-trivial quotient of $\Z_2^{g}$.
\end{prop}

\begin{proof}
The case $g=1$ is straightforward because $B_2(\RP)$ is isomorphic to the generalised quaternion group of order $16$~\cite[Theorem, p.~94]{VB}. In particular, $\Gamma_2(B_2(\RP))/\Gamma_3(B_2(\RP))\cong \Z_{2}$.
If $g\geq 2$, the residual nilpotence of $B_2(U_g)$ follows by arguing as in the proof of \repr{n=2orientablebound}\ref{it:resboundary1}, using the fact that $P_2(U_g)$ is residually $2$-finite~\cite[Theorem~1.1]{BGe}. Note that $B_2(U_2)$ is not nilpotent, since otherwise $P_2(U_2)$ would  be nilpotent, but we know from~\cite[Theorem~5.4]{GP} that this is not the case. If $g\geq 3$, the centre of the group $B_2(U_g)$ is trivial~\cite[Proposition~1.6]{PR}, and as in \rerem{notnilp}, we can prove that the group $\Gamma_2(B_2(U_g))/\Gamma_3(B_2(U_g))$ is non-trivial. To see that this group is a quotient of $\Z_2^{g}$, observe that for all $1\leq s<r\leq g$, we have $[\rho_r,\sigma_1^{-1}  \rho_s \sigma_1]=1$ in $B_2(U_g)$ by relation~(\ref{eq:rs2}), so $[\rho_r,\rho_s]=1$ in $B_2(U_g)/\Gamma_3(B_2(U_g))$. Thus $\Gamma_2(B_2(U_g))/\Gamma_3(B_2(U_g))$ is generated by the commutators of the form $[\rho_i,  \sigma_1]$, where $1\leq i\leq g$. Since $\sigma_1^2=1$ in $B_2(U_g)/\Gamma_3(B_2(U_g))$ by relation~(\ref{eq:rr}), these commutators are of order at most $2$.
\end{proof}

At this point, we may prove \repr{ornor} concerning the existence of a surjective homomorphism between $B_n(U_g)$ onto $B_n(\Sigma_{g-1})$.


\begin{proof}[Proof of \repr{ornor}]
First suppose that $g=1$, in which case $U_1=\RP$ and $\Sigma_0=\SP$. If $n=1$, $B_1(\SP)$ is trivial, and so there is clearly a surjection of $B_1(\RP)$ onto $B_1(\SP)$, and if $n=2$ then $B_2(\RP)$ is isomorphic to the generalised quaternion group of order $16$~\cite[Theorem, p.~94]{VB}, and $B_2(\SP)\cong \Z_2$ \cite[third theorem, p.~255]{FV}, and $B_2(\RP)$ surjects homomorphically onto $B_2(\SP)$. Finally, if $n\geq 3$, $(B_n(\RP))_{\text{Ab}}\cong \Z_2\oplus \Z_2$ by \reth{gam3nor}\ref{it:g1sigmaor}, and $(B_n(S^2))_{\text{Ab}}\cong \Z_{2(n-1)}$ by~\reth{gam3sph}\ref{it:gam1sp}, which implies that there is no surjection homomorphism from $B_n(\RP)$ onto $B_n(\SP)$. This proves the result in the case $g=1$.

Now assume that $g\ge 2$. If $n \ge 2$, $(B_n(U_g))_{\text{Ab}}\cong \Z^{g-1} \oplus \Z_2\oplus \Z_2$ by \reth{gam3nor}\ref{it:g1sigmaor} and $(B_n(\Sigma_{g-1}))_{\text{Ab}}= \Z^{2(g-1)} \oplus \Z_2$ by \reth{gam3closed}\ref{it:gam12g} while if $n=1$, $(B_1(U_g))_{\text{Ab}}\cong \Z^{g-1}  \oplus \Z_2$ and $(B_n(\Sigma_{g-1}))_{\text{Ab}}= \Z^{2(g-1)}$. Therefore it is not possible to surject $B_n(U_g)$ onto $B_n(\Sigma_{g-1}$ in this case..
\end{proof}

To prove \reth{gensurj} in the non-orientable case, we will require the following lemma for the Klein bottle.

\begin{lem}\label{lem:garrafe}
Let $\K$ be the Klein bottle. If $x$ and $y$ are elements of $\pi_1(\K)$, then $xyxy=y^{-1}xyx$ if and only if  $y=1$.
\end{lem}

\begin{proof}
If $y=1$ then the relation clearly holds. Conversely, suppose that there exist $x,y\in \pi_1(\K)$ that satisfy the relation. 
Recall that $\pi_1(\K)$ is isomorphic to the semi-direct product $\Z \rtimes \Z$, where the action is given by multiplication by $-1$. With respect to this decomposition, let $x=(a,b)$ and $y=(c,d)$. Substituting these elements into the given relation, the second coordinate yields $2b+2d=2b$, so $d=0$, and computing the first coordinate, we obtain $a+(-1)^{b}c+(-1)^{b}(a+(-1)^{b}c)= -c+a+(-1)^b(c+a)$. Therefore $-c=c$, so $c=0$ and hence $y$ is the trivial element of $\pi_{1}(\K)$.
\end{proof}

We now prove \reth{gensurj}\ref{it:main1b} in the non-orientable case, where $g>1$.


\begin{proof}[Proof of \reth{gensurj}\ref{it:main1b}, where $\Sigma=U_{g}$, and $g> 1$]
We study the three cases of the statement of \reth{gensurj}\ref{it:main1b} separately.
\begin{enumerate}

\item Let $n<m$ and $n\in \brak{1,2}$. Using \reth{gam3nor} and the fact that $\Gamma_1(B_1(U_{g}))/ \Gamma_2(B_1(U_{g}))=\Z^{g-1} \oplus \Z_2$, the arguments used in the proof of \repr{mingen} also apply to the non-orientable case. In particular, $G(B_{1}(U_{g}))=g$, $G(B_{2}(U_{g}))= g+1$ and $G(B_{m}(U_{g}))= g+2$ for all $m\ge 3$. It follows that there is no surjective homomorphism in this case. 

\item Suppose that $n>m$ and $m\in \brak{1,2}$. If $n\geq 3$, we have $\Gamma_2(B_n(U_{g}))= \Gamma_3(B_n(U_{g}))$ by \reth{gam3nor}\ref{it:g2sigmaor}. On the other hand, $B_m(U_{g})$ is residually nilpotent if $m=1$ (and is in fact residually $2$-finite, see for instance~\cite[proof of Theorem~4.5]{BGe}), or if $m=2$ by \repr{n=2orientablebound}. So 
$B_{n}(U_{g})$ cannot surject homomorphically onto $B_{m}(U_{g})$ if $n\geq 3$ and $m\in \brak{1,2}$. So assume that $n=2$ and $m=1$. If $g>2$, the result follows in a similar manner by noting that $\Gamma_2(B_{2}(U_g))/\Gamma_3(B_{2}(U_g))$ is finite by \repr{b2nor}, but that $\Gamma_2(\pi_1(U_g))/\Gamma_3(\pi_1(U_g))$ is infinite~\cite{LA}. So suppose that $g=2$, and assume that there exists a surjective homomorphism $\Phi\colon\thinspace B_2(U_2) \to  \pi_1(U_{2})$. Applying $\Phi$ to relation~(\ref{eq:rr}) with $i=1$, we have that $ \Phi(\rho_1) \Phi(\sigma_1) \Phi(\rho_1) \Phi(\sigma_1)= \Phi(\sigma_1)^{-1} \Phi(\rho_1)  \Phi(\sigma_1) \Phi(\rho_1) $. The relation given in the statement of \relem{garrafe} is therefore satisfied if we take $x=\Phi(\rho_1)$ and $y=\Phi(\sigma_1)$, and thus $\Phi(\sigma_1)=1$. It follows from relation~(\ref{eq:rs2}) that $\Phi(\rho_1)$ and $ \Phi(\rho_2)$ commute. We conclude that the image of $\Phi$ is an Abelian subgroup of $\pi_1(U_{2})$, and since this latter group is non Abelian, $\Phi$ cannot be surjective.

\item\label{it:proofbiii} If $n>m \geq 3$ and $n\neq 4$, it suffices to argue as in the proof of part~\ref{it:ngrmc} of \reth{gensurj}\ref{it:main1a}\ref{it:main1aii} given in \resec{suror}, and apply \reth{LinSurfcomb} in the non-orientable case.\qedhere
\end{enumerate} 
\end{proof}

The following theorem gives some results in the case where $g=1$.

\begin{thm}\label{th:rp}\mbox{} 
\begin{enumerate}
\item\label{it:rpa}
Suppose that one of the following conditions holds:
\begin{enumerate}
\item\label{it:rpai} $n<m$ and $n\in \brak{1,2}$.
\item\label{it:rpaii} $n>m\geq 2$.
\end{enumerate}
Then there is no surjective homomorphism from $B_{n}(\RP)$ to $B_{m}(\RP)$.

\item\label{it:rpb} Let $m,n\geq 2$, let $n'=2\lfloor \frac{n}{2} \rfloor$ and let $m'=2\lfloor \frac{m}{2} \rfloor$, where $\lfloor x \rfloor$ denotes the greatest integer less than or equal to $x$. Set $n'=2^{l}s$ and $m'=2^{k}r$, where $l,k\in \N$, and $s,r$ are odd integers. If $l>k$ then the image of any homomorphism $f\colon\thinspace B_{n}(\RP) \to B_{m}(\RP)$ is finite cyclic. In particular, there is no surjective homomorphism from $B_{n}(\RP)$ to $B_{m}(\RP)$ in this case.
\end{enumerate}
\end{thm}

\begin{rems}\mbox{}\label{rems:rp2}
\begin{enumerate}

\item\label{it:rp2b} If $n\geq 2$ then $B_{n}(\RP)_{\text{Ab}}\cong \Z_{2}\oplus \Z_{2}$ by \reth{gam3nor}\ref{it:g1sigmaor}, and since $B_{1}(\RP)\cong \Z_{2}$, we see that $B_{n}(\RP)$ surjects homomorphically onto $B_{1}(\RP)$ via Abelianisation.

\item Comparing the statement of \reth{gensurj}\ref{it:main1b} with that of \reth{rp}\ref{it:rpa}, if $n>2$, then  by the previous remark, there is a surjective homomorphism from $B_{n}(\Sigma)$ to $B_{2}(\Sigma)$ if $\Sigma=\RP$, which is not the case if $\Sigma=\Sigma_{g,b}$ or $U_{g+1}$, where $b,g\geq 1$. Further, if $n=4$ and $m=3$, we do not know whether there exists a surjective homomorphism from $B_{4}(\Sigma)$ to $B_{3}(\Sigma)$ if $\Sigma=\Sigma_{g,b}$ or $U_{g+1}$, where $b,g\geq 1$, but \reth{rp}\ref{it:rpaii} shows that there does not exist such a homomorphism if $\Sigma=\RP$.
\end{enumerate}
\end{rems}

\begin{proof}[Proof of \reth{rp}]\mbox{}
\begin{enumerate}
\item \begin{enumerate}
\item If $n<m$ and $n\in \brak{1,2}$, the conclusion follows from the fact that $B_{1}(\RP)\cong \Z_{2}$, $B_{2}(\RP)$ is isomorphic to the binary dicyclic group of order $16$, and if $m\geq 3$, $B_{m}(\RP)$ is infinite~\cite[Theorem, p.94]{VB}.

\item Assume that $n>m\geq 2$. If $m=2$, the result is a consequence of the fact that $B_{2}(\RP)$ is residually nilpotent, while $B_{n}(\RP)$ is not for all $n\geq 3$ by \reth{gam3nor}. Now suppose that $m\geq 3$. If $n\neq 4$, the result follows as in the proof of part~\ref{it:proofbiii} of \reth{gensurj}\ref{it:main1b} by applying \reth{LinSurfcomb} in the non-orientable case. We defer the proof of the case $n=4$ and $m=3$ to part~\ref{it:rpb}.
\end{enumerate}

\item Let $m,n\geq 2$, let $n'=2\lfloor \frac{n}{2} \rfloor$ and let $m'=2\lfloor \frac{m}{2} \rfloor$, and let $n'=2^{l}s$ and $m'=2^{k}r$, where $l,k\in \N$, and $s,r$ are odd integers. Let $\phi\colon\thinspace B_{n}(\RP) \to B_{m}(\RP)$ be a homomorphism. Consider the elements $a=\rho_{n}\sigma_{n-1} \cdots \sigma_{1}$ and $b=\rho_{n-1}\sigma_{n-2} \cdots \sigma_{1}$ of $B_{n}(\RP)$, where we use Van Buskirk's presentation of $B_{n}(\RP)$~\cite[p.~83]{VB}. By~\cite[Proposition~26]{GGagt}, $a$ (resp.\ $b$) is of order $4n$ (resp.\ $4(n-1)$). Let $x=a$ and $x'=b$ (resp.\ $x=b$ and $x'=a$) if $n$ is even (resp.\ is odd). Then $x$ is of order $4n'$, which in terms of the notation introduced in the statement, is equal to $2^{l+2}s$. Observe also that from the proof of~\cite[Theorem~6]{GGzeit}, $B_{n}(\RP)=\langle x,x' \rangle$. 
By~\cite[Theorem~4]{GGagt}, the (maximal) torsion of $B_{n}(\RP)$ (resp.\ of $B_{m}(\RP)$) is $4n$ and $4(n-1)$ (resp.\ $4m$ and $4(m-1)$), and so the maximal torsion in $B_{n}(\RP)$ that is a power of $2$ is equal to $2^{l+2}$ in $B_{n}(\RP)$, and is realised by $x^{s}$, and the maximal torsion in $B_{m}(\RP)$ that is a power of $2$ is equal to $2^{k+2}$. It follows that the order of $f(x^{s})$ is a divisor of $2^{k+2}$, in particular $f(x^{2^{k+2}s})=1$ in $B_{m}(\RP)$. Now $l\geq k+1$ by hypothesis, and so $1=(f(x^{2^{k+2}s}))^{2^{l-k-1}}=f(x^{2^{l+1}s})$. Since $x$ is of order $2^{l+2}s$, $x^{2^{l+1}s}$ is of order $2$, so is equal to the full twist braid $\Delta_{n}^{2}$ of $B_{n}(\RP)$, using the fact that $\Delta_{n}^{2}$ is the unique element of $B_{n}(\RP)$ of order $2$~\cite[Proposition~23]{GGagt}. We conclude that $\Delta_{n}^{2}\in \ker{f}$.

Now let $H=\ang{x,y}$, where $y=\Delta_{n}$ (resp.\ $y=\Delta_{n}a^{-1}$) if $n$ is even (resp.\ $n$ is odd). By~\cite[Proposition~15]{GGjlms}, $H$ is isomorphic to the dicyclic group $\operatorname{Dic}_{4n'}$ of order $4n'$, and the generators satisfy the relations $x^{n'}=y^{2}$ and $yxy^{-1}=x^{-1}$. Using once more the fact that $\Delta_{n}^{2}$ is the unique element of $B_{n}(\RP)$ of order $2$, we have $\Delta_{n}^{2}\in \ker{\phi} \cap H$. Further, $\Delta_{n}^{2}$ is central in $B_{n}(\RP)$~\cite[Proposition~6.1]{M}, and hence the restriction $f\left\lvert_{H}\right. \colon\thinspace H \to f(H)$ of $f$ to $H$ factors through the quotient $H/\langle \Delta_{n}^{2} \rangle$. But using the relations of $H$, this quotient is isomorphic to the dihedral group of order $2n'$, so $f(H)$ is a subgroup of $B_{m}(\RP)$ that is a quotient of $H/\langle \Delta_{n}^{2} \rangle$. Now the quotients of dihedral groups are either the trivial group, cyclic of order $2$ or dihedral, and since the braid groups of $\RP$ do not have dihedral subgroups~\cite[Theorem~5]{GGjlms}, it follows that $f(H)$ is either trivial or cyclic of order $2$, so $\ker{f\left\lvert_{H}\right.}$ is either equal to $H$, or is a subgroup of $H$ of index $2$. If $\ker{f\left\lvert_{H}\right.}$ is of index $2$ in $H$, then by analysing the images of $x$ and $y$ by a surjective homomorphism from $H$ to $\Z_{2}$, we see that either $\ker{f\left\lvert_{H}\right.}=\ang{x}$, or if $n'$ is even, additionally $\ker{f\left\lvert_{H}\right.}=\ang{x^{2},y}$, or $\ker{f\left\lvert_{H}\right.}=\ang{x^{2},xy}$. So if either $\ker{f\left\lvert_{H}\right.}$ is equal to $H$, or is an subgroup of $H$ of index $2$, we conclude from these possibilities that $x^{2} \in \ker{f}$. It follows again from the fact that $\Delta_{m}^{2}$ is the unique element of $B_{m}(\RP)$ of order $2$ that $f(x)\in \ang{\Delta_{m}^{2}}$. Since $f(x')$ is of finite order and $f(x)$ is central in $B_{m}(\RP)$, using the fact mentioned above in the first paragraph that $B_{n}(\RP)=\langle x,x' \rangle$, we see that the image of $f$ is finite cyclic as required. In particular, in the outstanding case of the proof of part~\ref{it:rpa}\ref{it:rpaii}, where $n=4$ and $m=3$, there is no surjective homomorphism from $B_{4}(\RP)$ to $B_{3}(\RP)$.\qedhere
\end{enumerate}

\end{proof}


%

\begin{cor}\label{cor:purenor} 
Let $g\geq 1$, and let $m,n\in \N$. 
Then there exists a surjective homomorphism of $B_n(U_g)$ onto $P_m(U_g)$ if and only if either $g=m=1$ or $n=m=1$.
\end{cor}

\begin{proof}
If $n=m=1$ and $g\geq 1$, the result is clear, and if $g=m=1$, the result follows from \rerems{rp2}\ref{it:rp2b}.
Conversely, suppose that there exists a surjective homomorphism $\Phi\colon\thinspace B_n(U_{g})\to P_m(U_{g})$. Then $\Phi$ induces a surjective homomorphism of the corresponding Abelianisations, but since $(P_m(U_{g}))_{\text{Ab}}\cong \Z^{(g-1)m} \oplus \Z_2^{m} $ using a presentation of $P_m(U_{g})$ (see~\cite[Theorem~3]{GGjpaa3} for instance), it follows from \reth{gam3nor}\ref{it:g1sigmaor} and the fact that $(B_{1}(U_{g}))_{\text{Ab}}\cong \Z^{g-1} \oplus \Z_2$ that $m=1$. So either $n=1$ or $g=1$, and thus the conclusion holds, 
or else $n>1$ and $g>1$, in which case we obtain a contradiction using \reth{gensurj}\ref{it:main1b}\ref{it:main1bii}.
\end{proof}

\section{Surjections between braid groups of orientable surfaces and symmetric groups} \label{sec:other}

In this section, we start by recalling \reth{LinTrans}, due to Ivanov~\cite{Iv}, about transitive representations of $B_n$ and $S_m$, where $n>m \geq 2$  (the definitions of primitive and transitive representations were given in \resec{intro}). We then prove~\reth{LinSurfaceTrans} that generalises \reth{LinTrans} to braid groups of compact, orientable surfaces. We shall assume that the surfaces are without boundary, but the results extend easily to the case with boundary.
In~\cite{Iv}, Ivanov gave some transitive, imprimitive representations of $B_n(\Sigma_g)$ in $S_n$, where $g\geq 1$ and $n\geq 3$. These representations have the property that their images are Abelian subgroups of $S_{n}$. We shall construct some transitive, imprimitive representations of $B_n(\Sigma_g)$ in $S_m$ whose images are non Abelian, so they are different from those of Ivanov. 

The following result is a variant of \reth{Lin} for transitive representations.

\begin{thm}[{\cite[Lemma~3]{Iv}}]\label{th:LinTrans}
Let $n>m \geq 2$, and let $\rho\colon\thinspace B_n  \to S_m$  be a transitive representation. Then one of the following statements holds:
\begin{enumerate}
\item\label{it:LinTransa} $\rho(\sigma_1)= \cdots =\rho(\sigma_{n-1})$, and this permutation is an $m$-cycle.
\item\label{it:LinTransb} if $n=4$ and $m=3$,
up to a suitable renumbering of the elements of the set $\brak{1, 2, 3}$, 
$\rho(\sigma_1)=\rho(\sigma_3)=(1,2)$ and $\rho(\sigma_2)=(2,3)$.
\end{enumerate}
\end{thm}

We can give an alternative  proof of \reth{LinTrans} using \reth{Lin} due to Lin.

\begin{proof}[Proof of \reth{LinTrans}]
Suppose that $n>m \geq 2$. If $m=2$ then $\im{\rho}$ is either trivial, which contradicts the transitivity hypothesis, or is equal to $S_{2}$, and  statement~\ref{it:LinTransa} holds. So suppose that $n>m \geq 3$, and assume that $n\neq 4$. Arguing as in the first part of the proof of \reth{LinSurfcomb} in \resec{suror}, it follows that $\rho(\sigma_1)= \cdots =\rho(\sigma_{n-1})$, and the fact that $\rho$ is transitive implies that the permutation $\rho(\sigma_1)$ of $S_{m}$ is an $m$-cycle, so once more statement~\ref{it:LinTransa} holds. Finally, assume that $(n,m)=(4,3)$. We claim that $\rho(\sigma_1)$ and $\rho(\sigma_2)$ have the same cycle type. To see this, first note that if one of $\rho(\sigma_1)$ or $\rho(\sigma_2)$ is equal to the identity permutation then the Artin relations imply that the other is also equal to the identity, which proves the claim in this case. So suppose that  one of these two elements is a transposition and the other is a $3$-cycle. Then $\rho(\sigma_1 \sigma_2 \sigma_1)$ and $\rho(\sigma_2 \sigma_1 \sigma_2)$ have opposite signatures, which yields a contradiction using the Artin relations, and proves the claim.
It follows in a similar manner that $\rho(\sigma_2)$ and $\rho(\sigma_3)$ have the same cycle type, hence $\rho(\sigma_1)$, $\rho(\sigma_2)$ and $\rho(\sigma_3)$ all have the same cycle type. By the transitivity hypothesis, they cannot be equal to the identity permutation, and they cannot be equal to the same transposition. So we are reduced to analysing the following two cases:
\begin{enumerate}[label=\textit{(\roman*)}]
\item $\rho(\sigma_1)$, $\rho(\sigma_2)$ and $\rho(\sigma_3)$ are transpositions. Since $\sigma_{1}$ and $\sigma_{3}$ commute, it follows that $\rho(\sigma_1)=\rho(\sigma_3)$, and the fact that $\rho(\sigma_1)$, $\rho(\sigma_2)$ and $\rho(\sigma_3)$ do not coincide implies that the condition given in part~\ref{it:LinTransb} is satisfied.

\item $\rho(\sigma_1)$, $\rho(\sigma_2)$ and $\rho(\sigma_3)$ are $3$-cycles. Using the Artin relations, it follows that $(\rho(\sigma_1))^{-1}\neq \rho(\sigma_2)$, so $\rho(\sigma_1)=\rho(\sigma_2)$. In a similar fashion, $\rho(\sigma_2)=\rho(\sigma_3)$, and thus the condition given in part~\ref{it:LinTransa} is satisfied.\qedhere
\end{enumerate}
\end{proof} 



We now recall the following result of~\cite{BGP} about the structure of the centraliser $C_{S_m}(u)$ of a permutation $u$ in $S_{m}$. Note that $C_{S_m}(u)$ is equal to the centraliser $C_{S_m}(\ang{u})$ of the subgroup $\ang{u}$ in $S_{m}$. If $k\in \N$, let $C_k$ denote the cyclic group of order $k$.

\begin{prop}[{\cite[Lemma~1.1]{BGP}}]\label{prop:berrickparis}
Let $u \in S_m$ be a permutation whose cycle type is equal to $(1)^{\ell_1} (2)^{\ell_2}\ldots(m)^{\ell_m}$, and let $I(u) = \setl{k \in \brak{1,2,\ldots,m}}{\ell_k > 0}$, so that $\Sigma_{k\in I(u)}\, k\ell_k =m$. 
Then the centraliser $C_{S_m}(u)$ of $u$ in $S_{m}$ is isomorphic to $\prod_{k\in I(u)}\, C_k^{\ell_k}\rtimes  S_{\ell_k}= \prod_{k\in I(u)}\, C_k \wr S_{\ell_k}$.
\end{prop}

 In the semi-direct product $\prod_{k\in I(u)}\, C_k^{\ell_k}\rtimes  S_{\ell_k}$ given in the statement of \repr{berrickparis}, the action of an element $\tau$ of $S_{\ell_k}$ is given by indexing the copies of $C_{k}$ by $\brak{1,\ldots, \ell_k}$, and by sending a given element of $C_{k}$ to the corresponding element of $C_{\tau(k)}$. Further, the partition associated with the cycle decomposition of $u$ is left invariant by the elements of $C_{S_m}(u)$. 


Let $n>m \geq 1$, and let $\rho_{n,m}\colon\thinspace  B_n(\Sigma_g) \to S_m$ be a representation.  Considering $B_{n}$ to be a subgroup of $B_n(\Sigma_g)$ induced by the inclusion of a topological disc in $\Sigma_{g}$, by abuse of notation, we also denote the restriction of $\rho_{n,m}$ to $B_n$ by $\rho_{n,m}$. We now prove \reth{LinSurfaceTrans} that generalises \reth{LinTrans}. 

\begin{proof}[Proof of \reth{LinSurfaceTrans}] Suppose that $n>m\geq 2$. If $m=2$ then $\im{\rho_{n,2}}$ is either equal to $\brak{\operatorname{\text{Id}}}$ or is isomorphic to $\Z_{2}$, and statement~\ref{it:LinSurfaceTransa} of the theorem holds. So assume that $n>m\geq3$, and suppose additionally that $n\neq 4$. Since $m\geq 3$, $\im{\rho_{n,m}} \neq \brak{\operatorname{\text{Id}}}$. By considering the composition of $\rho_{n,m}$ with the inclusion of $B_{n}$ in $B_{n}(\Sigma_{g})$, we see as in the proof of \reth{LinSurfcomb} in \resec{suror} that $\rho_{n,m}(\sigma_1)= \cdots =\rho_{n,m}(\sigma_{n-1})$.
We denote this common element of $S_{m}$ by $\sigma$. Relation~(\ref{eq:cs}) implies that $\sigma$ commutes with $\rho_{n,m}(a_i)$ and with $\rho_{n,m}(b_i)$ for all $1\leq i\leq g$. So if 
$\sigma$ is an $m$-cycle then $\rho_{n,m}(a_i)$ and $\rho_{n,m}(b_i)$ are powers of $\sigma$ for all $1\leq i\leq g$, in which case $\im{\rho_{n,m}}$ 
is generated by $\sigma$, and statement~\ref{it:LinSurfaceTransa} of the theorem holds. 
So assume that $\sigma$ is not an $m$-cycle. Then the decomposition of $\sigma$ as a product of disjoint cycles gives rise to a partition of the set $\{1, \ldots , m\}$ that is different from the set $\{1, \ldots , m\}$ itself  and that is invariant under the action of $\sigma$. Since $\rho_{n,m}(a_i)$ and $\rho_{n,m}(b_i)$ commute with $\sigma$ for all $1\leq i\leq g$, they also leave this partition invariant, 
and it follows from the hypothesis that $\rho_{n,m}$ is primitive that $\sigma$ is the identity permutation. Thus $\rho_{n,m}$ factors through the quotient $B_n(\Sigma_g)/\langle\!\langle \brak{\sigma_{1}, \ldots, \sigma_{n-1}}\rangle\!\rangle$ of $B_{n}(\Sigma_{g})$ by the normal closure $\langle\!\langle \brak{\sigma_{1}, \ldots, \sigma_{n-1}}\rangle\!\rangle$ of $\brak{\sigma_{1}, \ldots, \sigma_{n-1}}$ in $B_{n}(\Sigma_{g})$, and induces a homomorphism $\overline{\rho}_{n,m}\colon\thinspace B_n(\Sigma_g)/\langle\!\langle \brak{\sigma_{1}, \ldots, \sigma_{n-1}}\rangle\!\rangle \to S_m$. But from the proof of \repr{mingen}, $B_n(\Sigma_g)/\langle\!\langle \brak{\sigma_{1}, \ldots, \sigma_{n-1}}\rangle\!\rangle$ is isomorphic to $\Z^{2g}$. So $\im{\overline{\rho}_{n,m}}=\im{\rho_{n,m}}$ is non trivial and Abelian, and $\overline{\rho}_{n,m}$ is primitive. 
Since $\im{\overline{\rho}_{n,m}}$ is Abelian, any non-trivial element $u\in \im{\overline{\rho}_{n,m}}$ commutes with all of the elements of $\im{\overline{\rho}_{n,m}}$, from which we see that $\im{\overline{\rho}_{n,m}}$ is contained in the centraliser of $u$ in $S_{m}$. If $u$ is an $m$-cycle then $\im{\overline{\rho}_{n,m}}$ coincides with $C_{S_{m}}(u)$, which is equal to $\ang{u}$, and thus part~\ref{it:LinSurfaceTransa} of the statement holds. So assume that $\im{\overline{\rho}_{n,m}}$ contains no $m$-cycle. Then the cycle decomposition of $u$ contains a non-trivial cycle of length strictly less than $m$, so by \repr{berrickparis}, $C_{S_{m}}(u)$ is imprimitive. But since $\im{\overline{\rho}_{n,m}}\subset C_{S_{m}}(u)$, this implies that $\overline{\rho}_{n,m}$ is also imprimitive, which yields a contradiction. This argument also  implies   that $m$ has to  be prime, and that $u$ is an $m$-cycle. 
 This completes the proof of the case $n>m\geq 3$ and $n\neq 4$.

Finally, let $n=4$ and $m=3$. Suppose first that the restriction of the representation $\rho_{4,3} \colon\thinspace B_{4}(\Sigma_{g})\to S_{3}$ to $B_{4}$ is intransitive. Thus $\rho_{4,3}(B_4)$ is equal to a subgroup of $S_{3}$ of order $1$ or $2$, and in either case, it follows that $\rho_{4,3}(\sigma_1)=\rho_{4,3}(\sigma_2)=\rho_{4,3}(\sigma_3)$ using the Artin relations~\reqref{artin1} and~\reqref{artin2}. We denote this element by $\sigma$. As in the discussion of the previous paragraph of the case where $\sigma$ is not an $m$-cycle for $n>m\geq 3$ and $n\neq 4$, 
we obtain a contradiction. Therefore the restriction of the representation $\rho_{4,3}$ to $B_{4}$ is transitive, and by \reth{LinTrans}, we just have to consider the following two cases.

\begin{enumerate}[label=\textit{(\roman*)}]
\item $\rho_{4,3}(B_4)$ 
is generated by a $3$-cycle, and $\rho_{4,3}(\sigma_1)=\rho_{4,3}(\sigma_2)=\rho_{4,3}(\sigma_3)$. Using once more relation~(\ref{eq:cs}) of \reth{presbng}, we see that $\rho_{4,3}(a_i)$ and $\rho_{4,3}(b_{i})$ commute with the $3$-cycle $\rho_{4,3}(\sigma_1)$ for all $1\leq i \leq g$, so they are powers of $\rho_{4,3}(\sigma_1)$. Thus $\im{\rho_{4,3}}=\ang{\rho_{4,3}(\sigma_1)}$, and hence part~\ref{it:LinSurfaceTransa} of the statement holds. 

\item Up to a suitable renumbering of the elements of the set  $\brak{1, 2, 3}$, $\rho_{4,3}(\sigma_1)=\rho_{4,3}(\sigma_3)=(1,2)$ and $\rho_{4,3}(\sigma_2)=(2,3)$. Relation~(\ref{eq:cs}) of \reth{presbng} implies once more that $\rho_{4,3}(a_i)$ and $\rho_{4,3}(b_{i})$ commute with the elements $\rho_{4,3}(\sigma_2)$ and $\rho_{4,3}(\sigma_3)$ for all $1\leq i\leq g$. Since these transpositions generate $S_{3}$, it follows that the permutations $\rho_{4,3}(a_i)$ and $\rho_{4,3}(b_{i})$ belong to the centre of $S_{3}$  for all $1\leq i\leq g$, so are trivial. Hence part~\ref{it:LinSurfaceTransb} of the statement holds.\qedhere
\end{enumerate}
\end{proof}

\begin{rem}
Using the methods of the proof of \reth{LinSurfaceTrans} and the presentation given by~\cite[Proposition~3.1]{BGoG}, the statement of \reth{LinSurfaceTrans} also holds if the surface has boundary.
\end{rem}

We may obtain some information about an arbitrary representation $\rho_{n,m}\colon\thinspace B_n(\Sigma_g)\to S_m$ in a more general setting. 


\begin{prop}\label{prop:contraints}
Let $g\geq 1$, let $n>m\geq 2$, and assume that $(n,m)\neq (4,3)$.
Suppose that $\rho_{n,m}\colon\thinspace B_n(\Sigma_g)\to S_m$ is a homomorphism, and let $\rho_{n,m}(B_n)$ be the image of the subgroup $B_{n}$ of $B_n(\Sigma_g)$ under $\rho_{n,m}$.
\begin{enumerate}
\item\label{it:contraintsa} The subgroup $\rho_{n,m}(B_n)$ of $S_{m}$ is cyclic, and therefore $\rho_{n,m}(\sigma_1)=\cdots =\rho_{n,m}(\sigma_{n-1})$.
\item\label{it:contraintsb} The subgroup $\im{\rho_{n,m}}$
is contained in the centraliser $C_{S_m}(\rho_{n,m}(B_n))$ of $\rho_{n,m}(B_n)$ in $S_{m}$. 
This centraliser is described by \repr{berrickparis}. 
\item\label{it:contraintsc} There is an inclusion $\Gamma_3(B_n(\Sigma_g)) \subset \ker{\rho_{n,m}}$, so the homomorphism $\rho_{n,m}$ factors through the quotient $B_n(\Sigma_g)/\Gamma_3(B_n(\Sigma_g))$, and the subgroup $\im{\rho_{n,m}}$ is nilpotent of nilpotency degree at most $2$.
\end{enumerate}
\end{prop}

\begin{proof}
Parts~\ref{it:contraintsa} and~\ref{it:contraintsb} follow from the group presentation of $B_n(\Sigma_g)$ and by repeating the arguments given for instance in the proof of \reth{LinSurfaceTrans}. 
If $n\geq 3$, the first statement of part~\ref{it:contraintsc} is a consequence of part~\ref{it:contraintsa} and the fact that the subgroup $\Gamma_3(B_n(\Sigma_g))$ is isomorphic to the normal closure of the element $\sigma_1 \sigma_2^{-1}$ in $B_n(\Sigma_g)$~\cite[proof of Theorem~1(c)]{BGeG}, which is contained in $\ker{\rho_{n,m}}$ using part~\ref{it:contraintsa}. 
The second statement of part~\ref{it:contraintsc} then follows.
\end{proof}
  
In~\cite[p.~317]{Iv}, Ivanov gave some transitive, imprimitive representations of $B_n(\Sigma_g)$ in $S_n$ for $g\geq 1, n\geq 3$,
and he commented that `I do not know to what extent these examples exhaust the imprimitive representations'. All of the examples he proposed are representations whose images are Abelian. We now describe some imprimitive representations $\rho_{n,m}\colon\thinspace B_n(\Sigma_g) \to S_m$ whose images are non Abelian, so are different from those of Ivanov. 



\begin{exo}\label{ex:exo1}\mbox{}
\begin{enumerate}
\item By~\cite[eq.~(10)]{BGeG}, if $g\geq 1$ and $n\geq 3$, $B_n(\Sigma_g)/\Gamma_3(B_n(\Sigma_g))$ admits the following presentation:
\begin{enumerate}
\item[\textbf{generators:}] $a_1, b_1,\ldots,a_g, b_g$ and $\sigma$.
\item[\textbf{relations:}] $\sigma^{2(n-1+g)}=1$, and the elements of $\brak{a_1, b_1,\ldots,a_g, b_g,\sigma}$ commute pairwise, except for the pairs $(a_i, b_i)_{i=1,\ldots,g}$, for which $[a_1,b_1]=\cdots=[a_g,b_g]=\sigma^2$.
\end{enumerate}
Let $n>2$ be even, and let $g=1$. From the above presentation, we have:
\begin{equation*}
B_n(\T)/\Gamma_3(B_n(\T))=\setangl{a_1, b_1, \sigma}{[a_1,\sigma]=[b_1,\sigma]=1,\,  [a_1, b_1]=\sigma^2,\,  \sigma^{2n}=1}.
\end{equation*}
We define a map $\theta\colon\thinspace B_n(\T)/\Gamma_3(B_n(\T)) \to S_8$ on the generators of $B_n(\T)/\Gamma_3(B_n(\T))$ by: 
\begin{equation*}
\theta (a_1)=(1, 3)(2, 4),\,  \theta (b_1)=(1, 5)(2, 6)(3, 7)(4, 8) \; \text{and}\;   \theta (\sigma)=(1, 2, 3, 4)(5, 6, 7, 8).
\end{equation*}
It is straightforward to check that $\theta$ respects the relations of $B_n(\T)/\Gamma_3(B_n(\T))$, the equality $(\theta(\sigma))^{2n}=1$ being a consequence of the fact that $2n$ is divisible by $4$, so $\theta$ is a homomorphism.
If $p\colon\thinspace B_n(\T) \to B_n(\T)/\Gamma_3(B_n(\T))$ is the canonical projection, then the representation $\theta\circ p\colon\thinspace B_n(\T) \to S_{8}$ is transitive, and it is imprimitive since the non-trivial partition  $\brak{\brak{1,2,3,4}, \brak{5,6,7,8}}$ is preserved by the subgroup $\im{\theta\circ p}$ of $S_{8}$.  
This is perhaps the simplest example of an imprimitive representation $\rho_{n,m}\colon\thinspace  B_n(\Sigma_g) \to S_m$ whose image is non Abelian. In particular, if we take $n=8$, we obtain a transitive, imprimitive representation of $B_8(\T)$ in $S_8$ whose image is non Abelian, so it is not included in the examples of~\cite{Iv}.

\item

 
If $g+n$ is odd and $m=2^{g+2}$, \rex{exo1} may be generalised to construct a homomorphism $\theta_g \colon\thinspace B_n(\Sigma_g)/\Gamma_3(B_n(\Sigma_g))\to S_m$ such that $\im{\theta_{g}}$ is non Abelian. Composing $\theta_g$ with the projection $p_g \colon\thinspace B_n(\Sigma_g) \to B_n(\Sigma_g)/\Gamma_3(B_n(\Sigma_g))$, we thus obtain a homomorphism $\theta_g\circ p_g \colon\thinspace B_n(\Sigma_g) \to  S_m$ such that $\im{\theta_{g}\circ p_{g}}$ is non Abelian. To do so, first let us denote the image by $\theta_g$ of the element $\sigma\in B_n(\Sigma_g)/\Gamma_3(B_n(\Sigma_g))$ given in \rex{exo1} by $\overline{\sigma}\in S_m$. By \repr{contraints}, $\im{\theta_g}$ is contained in the centraliser of $\overline{\sigma}$ in $S_{m}$, which is described in \repr{berrickparis}. Our strategy is to make use of the structure of this centraliser to construct imprimitive representations whose images are non Abelian. 
In \rex{exo1}, the image of $\theta$ is isomorphic to $(\Z_4 \oplus \Z_4)\rtimes \Z_2$, and is the centraliser of $\overline{\sigma}=(1_4,1_4 ; 0_2)$ given by \repr{berrickparis}. In the general case, $n+g$ is odd, $m=2^{g+2}$, and the centraliser of $\overline{\sigma}$ is isomorphic to $\Z_4^{2^g}\rtimes \Z_{2^g}$. We now give two examples of this construction, one in the case where $g$ is odd, and in the other in the case where $g$ is even. 


\begin{enumerate}
\item Suppose that $g=3$, so $m=32$, and $n\geq 4$ is even. Consider $\Z_4^{8}\rtimes S_{8}$, which is interpreted as a subgroup of $S_{32}$. Define the homomorphism $\theta_3 \colon\thinspace B_n(\Sigma_3)/\Gamma_3(B_n(\Sigma_3)) \to S_{32}$ by $\theta_3(a_1)=(2,0,2,0,2,0,2,0)$, $\theta_3(a_2)=(2,2,0,0,2,2,0,0)$, $\theta_3(a_3)=(2,2,2,2,0,0,0,0)$, $\theta_3(\sigma)=(1,1,1,1,1,1,1,1)$, regarded as elements of $S_{32}$, where each factor $1$ denotes the cyclic permutation of length $4$ associated to the four integers corresponding to these four positions, $2$ denotes the square of this cyclic permutation, and $0$ denotes the identity permutation associated to these four integers.   
Finally, let  $\theta_3(b_1)=(1,2)(3,4)(5,6)(7,8)$, $\theta_3(b_2)=(1,3) (2,4)(5,7)(6,8)$, $\theta_3(b_3)=(1,5) (2,6)(3,7)(4,8)$, all regarded as elements of $S_{8}\subset \Z_4^{8}\rtimes S_{8}$.
In terms of explicit elements of $S_{32}$, we have:
\begin{align*}
\theta_3(a_1) =& (1, 3)(2, 4)(9, 11)(10, 12)(17, 19)(18, 20)(25, 27)(26, 28)\\   
\theta_3(a_2) =& (1, 3)(2, 4)(5, 7)(6, 8)(17, 19)(18, 20)(21, 23)(22, 24)\\
\theta_3(a_3) =&  (1, 3)(2, 4)(5, 7)(6, 8)(9, 11)(10, 12)(13, 15)(14, 16)\\
\theta_3(b_1) =& (1, 5)(2, 6)(3, 7)(4, 8) (9, 13 )(10, 14)(11, 15)(12, 16)(17, 21 )(18, 22)(19, 23)(20, 24)\cdot\\
& (25, 29)(26, 30)(27, 31)(28, 32)\\
\theta_3(b_2) =&(1, 9 )(2, 10)(3, 11)(4, 12)    (5, 13 )(6, 14)(7, 15)(8, 16)(17, 25 )(18, 26)(19, 27)(20, 28)\cdot\\
&  (21, 29)(22, 30)(23, 31)(24, 32)\\
\theta_3(b_3) =&(1, 17 )(2, 18)(3, 19)(4, 20) (5, 21 )(6, 22)(7, 23)(8, 24)(9, 25 )(10, 26)(11, 27)(12, 28)\cdot\\
&  (13, 29)(14, 30)(15, 31)(16, 32)\\
\theta_3(\sigma) =&(1, 2, 3, 4)(5, 6, 7, 8)(9, 10, 11, 12)(13, 14, 15, 16)(17, 18, 19,20)(21, 22, 23, 24)\cdot\\
&  (25, 26, 27, 28)(29, 30, 31, 32).
\end{align*}
Using these expressions, we may check that $[\theta_3(a_i),\theta_3(a_j)]=[\theta_3(b_i),\theta_3(b_j)]= [\theta_3(a_i),\theta_3(b_j)]$ for all $1\leq i<j\leq 3$, and $[\theta_3(a_l),\theta_3(b_l)]= (\theta_3(\sigma))^2$ for all $1\leq l\leq 3$, so $\theta_3$ is a homomorphism.
   

\item Now suppose that $g=2$, so $m=16$, and $n\geq 3$ is odd. Consider the subgroup $\Z_4^{4}\rtimes S_{4}$, which we interpret  as  a subgroup of the symmetric group $S_{16}$. We define a homomorphism $\theta_{2,1}\colon\thinspace B_n(\Sigma_2)/\Gamma_3(B_n(\Sigma_2)) \to S_{16}$ 
as follows. Let $\theta_{2,1}(a_1)=(2,0,2,0)$ and $\theta_{2,1}(a_2)=(2,2,0,0)$ in $S_{16}$, 
 where as in the previous example, each factor $1$ denotes the cyclic permutation of length $4$ associated to the integers corresponding to the these four positions, $2$ is the square of this cyclic   permutation, and $0$ is the identity permutation associated to these four integers. We also take $\theta_{2,1}(b_1)=(1,2) (3, 4 )\in S_4$, $\theta_{2,1}(b_2)=(1,3) (2,4)\in S_4$,
  and $\theta_{2,1}(\sigma)=(1,1,1,1) \in S_{16}$.  
Since $n+2$ is odd, $\theta_{2,1}$ defines a homomorphism. In $S_{16}$, the elements are given explicitly by:
\begin{align*}
\theta_{2,1}(a_1) &= (1, 3)(2, 4)(9, 11)(10, 12)\\
\theta_{2,1}(a_2) &= (1, 3)(2, 4)(5, 7)(6, 8)\\
\theta_{2,1}(b_1) &=(1, 5 )(2, 6)(3, 7)(4, 8) (9, 13 )(10, 14)(11, 15)(12, 16)\\
\theta_{2,1}(b_2) &=(1, 9 )(2, 10)(3, 11)(4, 12)(5, 13 )(6, 14)(7, 15)(8, 16)\\
\theta_{2,1}(\sigma) &=(1, 2, 3, 4)(5, 6, 7, 8)(9, 10, 11, 12)(13, 14, 15, 16).
\end{align*}
\end{enumerate}
\end{enumerate}
\end{exo}

\begin{exo}\label{ex:exo2}
    Let $n>2$, and consider the group $\Z_{2n}^n\rtimes S_{n}$ seen as subgroup of $S_{2n^2}$. Let $\theta\colon\thinspace B_{n}(\T)/\Gamma_3(B_{n}(\T)) \to S_{2n^2}$ be the homomorphism defined by $\theta(a_1)=(a, a+2, a+4,\ldots, a-2)\in \Z_{2n}^n$ for $a$ any element of $\Z_{2n}$, $\theta(b_1)=(1, 2, \ldots, n)\in S_n$, and $\theta(\sigma)=(1_{2n},\ldots,1_{2n})\in  \Z_{2n}^n$. It follows that $\theta(\sigma)$ is of order $2n$, $\theta(\sigma)$ commutes with $\theta(a_1)$ and $\theta(b_1)$, and that $\theta([a_1,b_1])=\theta(\sigma)^2$. 
The image of $\theta$ is the subgroup generated by $\brak{\theta(a_1),  \theta(b_1), \theta(\sigma)}$ and the image is non Abelian.
\end{exo} 
 
 
We conclude this paper with the following remarks.

\begin{rems}\mbox{}
\begin{enumerate}
\item The construction of \rex{exo2} also enables us to obtain an example of a homomorphism $\theta\colon\thinspace B_n(\T)/\Gamma_3(B_n(\T)) \to S_m$, where $n>m$, and the order of $\theta(\sigma)$ is equal to $2n$. First note that if $l\geq 3$ and $l$ divides $n$, then it follows from the presentation given at the beginning of \rex{exo1} that the map $\tau_{l} \colon\thinspace B_n(\T)/\Gamma_3(B_n(\T))\to B_l(\T)/\Gamma_3(B_l(\T))$ defined by sending the generators $a_{1}$, $b_{1}$ and $\sigma$ of $B_n(\T)/\Gamma_3(B_n(\T))$ to the generators $a_{1}$, $b_{1}$ and $\sigma$ respectively of $B_l(\T)/\Gamma_3(B_l(\T))$ extends to a (well-defined) surjective homomorphism. Let $n=3\cdot 5 \cdot 7 \cdot 11=1155$. For $l=3,5,7,11$, let $\theta_l \colon\thinspace B_l(\T)/\Gamma_3(B_l(\T)) \to S_{2l^{2}}$ be the homomorphism given as in \rex{exo2}, and let $\theta \colon\thinspace B_{n}(\T)/\Gamma_3(B_{n}(\T)) \to S_{408}$ be defined by $\theta(x)=(\theta_{3}\circ \tau_{3}(x), \theta_{5}\circ \tau_{5}(x), \theta_{7}\circ \tau_{7}(x), \theta_{11}\circ \tau_{11}(x))\in S_{18}\times S_{50}\times S_{98} \times S_{242}$ for all $x\in B_n(\T)/\Gamma_3(B_n(\T))$. Interpreting $S_{18}\times S_{50}\times S_{98} \times S_{242}$ as a subgroup of $S_{408}$, we may thus take $m=408$, and the element $\theta(\sigma)$ is of order $2310$. 


\item Let $g\geq 1$, and let $G$ be the group that admits the following presentation:
\begin{enumerate}
\item[\textbf{generators:}] $a_1, b_1,\ldots,a_g, b_g$ and $\sigma$.
\item[\textbf{relations:}] $\sigma^{2(1+g)}=1$, and the elements of $\brak{a_1, b_1,\ldots,a_g, b_g,\sigma}$ commute pairwise, except for the pairs $(a_i, b_i)_{i=1,\ldots,g}$, for which $[a_1,b_1]=\cdots=[a_g,b_g]=\sigma^2$.
\end{enumerate}
Observe that this is the group obtained by taking $n=2$ in the presentation of the quotient $B_{n}(\Sigma_g)/\Gamma_{3}(B_{n}(\Sigma_g))$ given in \rex{exo1} (we suspect that $B_{2}(\Sigma_g)/\Gamma_{3}(B_{2}(\Sigma_g))$ is not isomorphic to $G$ in this case). Using the presentation of $B_{2}(\Sigma_g)$ given by \reth{presbng}, the map $\rho\colon\thinspace B_{2}(\Sigma_g) \to G$ given by sending the generators $a_i$, $b_i$ and $\sigma_{1}$ of $B_{2}(\Sigma_g)$ to the generators $a_i$, $b_i$ and $\sigma$ respectively of $G$ for all $1\leq i\leq g$ may be seen to extend to a well-defined surjective homomorphism. To check that relation~\reqref{tot} is respected by $\rho$, note that in $G$:
\begin{equation*}
\prod_{i=1}^g [a_i^{-1},b_i]=\prod_{i=1}^g a_i^{-1}[b_i,a_i]a_i=\prod_{i=1}^g a_i^{-1}\sigma^{-2}a_i=\sigma^{-2g}=\sigma^2\; \text{since}\; \sigma^{2g+2}=1.
\end{equation*}
Hence $\rho$ induces a surjective homomorphism $\overline{\rho}\colon\thinspace B_{2}(\Sigma_g)/\Gamma_{3}(B_{2}(\Sigma_g)) \to G/\Gamma_{3}(G)$. Using the presentation of $G$ and the fact that $\Gamma_{2}(G)$ is the normal closure in $G$ of the commutators of the generators of $G$, we see that $\Gamma_{2}(G)=\ang{\sigma^{2}}$, and thus $\Gamma_{3}(G)$ is trivial. Therefore $\overline{\rho}$ is a surjective homomorphism from $B_{2}(\Sigma_g)/\Gamma_{3}(B_{2}(\Sigma_g))$ to $G$. Observe that $\overline{\rho}$ is not an isomorphism if $g=1$ because $\Gamma_{2}(B_{2}(\T))/\Gamma_{3}(B_{2}(\T))\cong \Z_2^3$ by \reth{gam3closed}\ref{it:gam23gn2}, and $\Gamma_{2}(G)/\Gamma_{3}(G)=\ang{\sigma^{2}} \cong \Z_2$. The construction of \rex{exo1} may be applied to $G$ if $g$ is odd, and composing with $\overline{\rho}$, shows that it may also be extended to the case $n=2$ to yield a representation of $B_2(\Sigma_g)$ in $S_{2^{g+2}}$ whose image is non Abelian.
\end{enumerate}
\end{rems}

\end{document}